\documentclass[twoside,10pt, a4paper]{article}
\makeatletter
\def\blfootnote{\gdef\@thefnmark{}\@footnotetext}
\makeatother
\usepackage[ansinew]{inputenc}
\usepackage[T1]{fontenc}
\usepackage{amsmath}
\usepackage{amsthm}
\usepackage{amscd, amsfonts, amssymb}
\usepackage[all]{xy}
\usepackage{mathrsfs}
\usepackage[dvips]{graphicx}
\usepackage{epic, eepic}
\usepackage[colorlinks = true, citecolor = purple, urlcolor  = purple, linkcolor = violet]{hyperref}
\usepackage{xcolor}
\usepackage{marvosym}
\usepackage[framemethod=tikz]{mdframed}

\usepackage{titling}
\usepackage{titlesec}
\titleformat*{\section}{\normalfont\Large\color{blue!80!black}}
\titleformat*{\subsection}{\normalfont\large\color{blue!80!black}}
\titleformat*{\subsubsection}{\normalfont\normalsize\color{blue!80!black}}
\usepackage[dotinlabels]{titletoc}
\usepackage{abstract}
\usepackage{helvet}         
\usepackage{courier}        
\usepackage{fix-cm}
\usepackage{BOONDOX-calo}

\numberwithin{equation}{section}

\theoremstyle{definition}
\newtheorem{dfn}{Definition}[section]
\newtheorem{example}{Example}[section]

\newtheorem{remark}{Remark}[section]
\theoremstyle{definition}
\newtheorem{prop}{Proposition}[section]
\newtheorem{thm}[prop]{Theorem}

\newtheorem{corollary}[prop]{Corollary}
\newtheorem{lem}[prop]{Lemma}
\newtheorem*{thm*}{Theorem}

\renewcommand{\fam}[1]{\mathsf{#1}}

\newcommand{\Z}{\mathbb{Z}}
\newcommand{\Q}{\mathbb{Q}}
\newcommand{\N}{\mathbb{N}}
\newcommand{\F}{\mathbb{F}}

\newcommand{\rrangle}{\,\rangle\!\!\!\rangle\,}
\newcommand{\llangle}{\langle\!\!\!\langle\,}

\newcommand{\Mod}{\mathsf{Mod}}

\newcommand{\azu}{\mathtt{azu}}
\newcommand{\ram}{\mathtt{ram}}
\newcommand{\branch}{\mathtt{branch}}

\newcommand{\Spec}{\mathrm{Spec}}

\newcommand{\Specm}{\mathrm{Specm}}

\newcommand{\Max}{\mathrm{Max}}

\newcommand{\eps}{\boldsymbol{e}}

\newcommand{\pp}{\mathfrak{p}}

\newcommand{\oo}{\mathfrak{o}}
\newcommand{\qq}{\mathfrak{q}}

\newcommand{\mm}{\mathfrak{m}}

\newcommand{\ff}{\boldsymbol{f}}

\newcommand{\eKW}{\mathbf{W}^{1/n}}

\newcommand{\eJac}{\mathcal{J}}

\newcommand{\W}{W}

\newcommand{\OO}{\mathcal{O}}

\renewcommand{\AA}{A}

\newcommand{\MM}{M}

\newcommand{\cent}{\mathfrak{Z}}

\newcommand{\Xscr}{\mathbcal{X}}

\newcommand{\xibf}{\boldsymbol{\xi}}

\newcommand{\Inf}{\mathcal{I\!\!\!\!I}}
\newcommand{\Ecal}{\mathcal{E}}

\DeclareMathOperator{\id}{id}

\DeclareMathOperator{\Ann}{Ann}

\DeclareMathOperator{\End}{End}
\DeclareMathOperator{\Hom}{Hom}
\DeclareMathOperator{\Ext}{Ext}

\usepackage{hyperref}

\begin{document}
\pretitle{\begin{flushleft}\LARGE \scshape
} 
\posttitle{\par\end{flushleft}
\rule[8mm]{\textwidth}{0.1mm}
}
\preauthor{\begin{flushleft}\Large \scshape
\vspace{-5mm}}
\postauthor{\end{flushleft}\vspace{-8mm}}                                                                                                                                                                                                                                                                                                                                                                                                                                                           
\title{Torsion points on Elliptic Curves and the Jackson Space}
\author{\large Daniel Larsson }

\date{}

\maketitle
\vspace{-0.5cm}
\begin{abstract}  
\noindent In this paper we will use a particular non-commutative scheme to, among other things, study the ramification properties of the field of $p$-torsion points on an elliptic curve and its reduction properties. Also, we show that this non-commutative space also allow us to use the $p$-torsion points to ``parametrise'' classes in the $p$-torsion of the Brauer group of the base field. 
\end{abstract}
\section{Introduction}

Let\blfootnote{\textit{e-mail:}  \href{mailto:daniel.larsson@usn.no}{ daniel.larsson@usn.no}}  $K$ be a local field with ring of integers $\oo_K$ and $E_{/K}$ an elliptic curve over $K$ with minimal regular model $\Ecal_{/\oo_K}$. Arithmetically it is interesting to study the specialisation (reduction) of $E$ to the special fibre of $\Ecal$. In this paper we will use a non-commutative deformation theory (see \cite{Laudal_Def} or \cite{EriksenLaudalSiqveland}) and non-commutative algebraic geometry to do exactly that. 

A natural question is what non-commutative deformation theory and non-commutative algebraic geometry can bring to the table in this context. As deformation theory is the local study of an object in the moduli space of all such objects, deformation theory can give information connected to the reduction of these objects to characteristic $p$. Extension to non-commutative base rings as versal deformation spaces, allows us to study the reduction properties of a family of objects \emph{at the same time}. In other words, how the objects deform (reduce) together as a family. This simultaneous deformations are measured by the size of certain $\Ext^1$-groups. In our case the family of objects are rational points on non-commutative spaces constructed by using $p$-torsion points on $E$. 

The starting point in the present approach is the following. Let $p$ be a prime. We will assume through most of the paper that $p\geq 3$, but this is not necessary at this point. Under the assumption that the $p$-torsion group $E[p]$ has a $K$-rational point, the field $K(E[p])$ generated by the coefficients of all $p$-torsion points is known to fit in the following tower of fields
$$\xymatrix{K\ar@{-}[r]&K(\zeta_p)\ar@{-}[r]&K(E[p]),}$$where $\zeta_p$ is primitive $p$-th root of unity. In fact, $K(E[p])/K(\zeta_p)$ is a Kummer extension (this is well-known but see \cite{Yasuda_Kummer} for one reference)
$$K(E[p])=K(\zeta_p)[\sqrt[p]{x}]=K(\zeta_p)[T]/(T^p-x), \qquad x\in K.$$Put $F:=K(E[p])$. Clearly, $F$ encodes interesting arithmetic properties associated with $E$ and its $p$-torsion points. From the field $F$ we will construct two non-commutative algebras $\eJac_x$ and $\Inf_{\!\xi}$ (the description of $\xi$ comes below), with associated non-commutative schemes, $\Xscr_{\eJac_x}$ and $\Xscr_{\Inf_{\!\xi}}$, respectively, using non-commutative deformation theory. Using these noncommutive algebras/schemes we will study the specialisation of $E$ to the special fibre. For simplicity we will, in this introduction denote either of $\Xscr_{\eJac_x}$ and $\Xscr_{\Inf_{\!\xi}}$, by $\Xscr$. In this paper we will use a ``soft'' approach to non-commutative algebraic geometry, more meant to be suggestive than actually essential in this particular instance. For a more comprehensive discussion see \cite{LarssonAritGeoLargeCentre} and the references therein. 

Let me indicate the construction of $\eJac_x$ since this is slightly more simple. This algebra is associated with the reduction modulo a prime $\ell\neq p$. The algebra $\Inf_{\!x}$ is used for the reduction modulo $p$ itself. 

Suppose $\oo_K\to B\to A$ is a general cyclic extension of an $\oo_K$-algebra $B$ and let $G$ be its automorphism group. Take an element $\sigma\in G$. The operator $\Delta:=a(\id-\sigma)$ for $a\in A$ acts on $A$ and satisfies the generalised Leibniz rule $\Delta(bc)=\Delta(b)c+\sigma(b)\Delta(c)$. Form the left $A$-module $A\cdot\Delta$. This can be endowed with a nonassociatve algebraic structure turning it into a generalised Lie algebra called a hom-Lie algebra (see section \ref{sec:twisted_derivations} for more details). Akin to the enveloping algebra of a Lie algebra, so can one construct an associative ``enveloping'' algebra of a hom-Lie algebra. The algebra $\eJac_x$ is then defined as the subalgebra of the ``enveloping'' algebra of the hom-Lie algebra $A\cdot\Delta$. In fact, $\eJac_x$ is a ``$q$-deformation'' of the enveloping algebra of $\mathfrak{sl}_2$ (see \cite{LaSi}).  

\begin{remark}
Notice that $A\cdot\Delta$ captures information of the group action of $G$ on $A$. This transfers to naturally to $\eJac_x$. 
\end{remark}

It turns out that $\eJac_x$ (and $\Inf_{\!\xi}$) satisfies some desirable ring-theoretic and homological properties. For instance, the algebra (or its associated non-commutative scheme) is noetherian, normal, Auslander-regular, and Cohen--Macaulay (CM). The definition of ``Auslander-regularity'' is not necessary here beyond it including a condition of finite global dimension. Hence this is indeed a generalisation of regularity in the commutative context. However, rings can be Auslander-regular and still behave ``singular'' in the non-commutative realm. Indeed, as we will see an Auslander-regular ring can have a singular centre. Similarly, a non-commutative ring can be CM, but have a centre which is not. On the other hand, the algebras we will consider here, will have CM centres. I refer to the literature for the definition of Auslander-regular and Cohen--Macaulay in the non-commutative case (a definition can be found in \cite{LarssonKummerWitt}, for instance).

The paper is organised as follows. Section \ref{sec:twisted_derivations} is preliminary and contains, for the reader's benefit, the constructions of algebras related to twisted derivations. In particular, it includes the construction of an ``infinitesimal enveloping algebra'', $\mathbf{I}_q$, that will crucially be used later. This section overlaps, by necessity, some discussions from \cite{LarssonKummerWitt}. However, there are some differences in presentation.  

Section \ref{sec:Kummer_Witt} introduces the so-called Kummer--Witt algebra $\eKW_x$. This is defined to be the ``enveloping algebra'' of the hom-Lie algebra $A\cdot \Delta$ where $A$ is a Kummer extension of $B$. This section also lists a number of ring-theoretic properties satisfied by $\eKW_x$ and $\mathbf{I}_q$.  

Section \ref{sec:nc_space} introduces the notion of non-commutative algebraic space that we will use. Included is also a definition of rational points on such a space. The section following section \ref{sec:nc_space} contains a recollection of polynomial identity algebras and the definition of $\eJac_x$ and a description of its centre and some of the ring-theoretic properties it enjoys. We also describe the non-commutative space $\Xscr_{\eJac_x}$ and its set of one-dimensional rational points together with their infinitesimal structure. 

In section \ref{sec:ram} we discuss ramification in cyclic extensions and its connection to $\eJac_x$. Then in section \ref{sec:torsionelliptic}, which is the main part of the paper, we consider the reduction properties of elliptic curves and relate these to properties of $\eJac_x$ and $\Inf_{\!\xi}$ and their associated spaces. In the case where $F=K(E[\ell])$ and $\ell\neq p$ we use $\eJac_x$. However, $\eJac_x$ does not give enough information in the case $\ell=p$ and this is where $\Inf_{\!x}$ comes in. Let $f$ be the conductor of $E$ and $d$ the number of irreducible components of the special fibre $\bar E$ which are defined over $k$ (the residue class field of $K$). We now take $\xi=\zeta_{f+d}$ and $\Inf_{\!\zeta_{f+d}}$. This choice might seem ad hoc at first glance but it turns out to be a very fruitful choice.  

The reduction properties of $E$ in the case where $\ell=p$ are studied via a family of one-dimensional rational points on $\Xscr_{\Inf_{\!\zeta_{f+d}}}$. We pay particular attention to additive reduction. It turns out that the only possible values for $p$ in the case where $K$ is absolutely unramified are $2, 3, 5$ and $7$. En route to the construction of the one-dimensional points alluded to above, we introduce a set of elements $\theta_i\in F=K(E[p])$. These are used in the last section to construct hyperplane arrangements in $\Xscr_{\Inf_{\!x}}$ and Brauer-classes in $\mathrm{Br}(K)$.  

It is not an understatement to say that the study of $K$-rational $p$-torsion points and their associated $p$-torsion fields $K(E[p])$ is an active field of study. For instance, the recent paper \cite{FreitasKraus} by Freitas and Kraus gives a complete description of the degree $[\Q_\ell(E[p]):\Q_\ell]$ when $\ell\neq p\geq 3$; in \cite{BandiniPaladino} Bandini and Paladino studies the generators of $K(E[p])$ where $K$ is any field of characteristic not equal to $2$ or $3$. The latter paper also indicate the complexity of the $p$-torsion field with several examples. It is my hope that the present paper could also be used to study the fields $K([p])$ as well as the $p$-torsion points themselves. 

\subsubsection*{Acknowledgements} 
Thanks to \'Alvaro Lozano-Robledo and Chris Wuthrich for answering questions about elliptic curves, their torsion points and reduction properties. Arnfinn Laudal deserves (as always) a special thanks for encouragement and all that.   
\subsection*{Notation}
We will adhere to the following notation throughout.
\begin{itemize}
	\item All associative rings are unital. Unless otherwise stated all rings $A, B, k$ e.t.c., are associative. 
		\item All modules are \emph{left} modules unless otherwise explicitly specified. The groupoid of (left) $A$-modules up to isomorphism is denoted $\Mod(A)$.
	\item The notation $\Max(A)$ denotes the \emph{set} of maximal ideals, while $\Specm(A)$ denotes the maximal spectrum of $A$, i.e., $\Max(A)$ together with the Zariski topology.
	\item $\cent(A)$ denotes the centre of $A$. 
	\item For $\pp$ a prime in $A$, $k(\pp)$ denotes the residue class field of $\pp$. 
\end{itemize}

\section{Algebras related to twisted derivations}\label{sec:twisted_derivations}
We should probably say a few things about the construction of the family of algebras $\eJac_x(r)$, otherwise what follows will feel completely pulled from the mist. Since the construction is spelled out in detail in \cite{LarssonKummerWitt} we will deliberately be as brief as possible. 

Throughout $\oo$ is a commutative ring and $A$ commutative $\oo$-algebra.  
\subsection{hom-Lie algebras}
Suppose then that $\sigma$ is a $\oo$-algebra endomorphism on $\AA$. Then a \emph{twisted $\oo$-derivation} (or \emph{$\sigma$-derivation}) on $\AA$ is an operator
$$\Delta(ab)=\Delta(a)b+\sigma(a)\Delta(b),\quad a,b\in \AA.$$ 
\begin{example}
The canonical example of a $\sigma$-derivation is a map $\AA\to\AA$ on the form
$$\Delta:=a(\id-\sigma):\,\,\AA\to\AA, \quad a\in\AA.$$In fact, for many algebras these types of maps are the only $\sigma$-derivations available. 
\end{example}
Under certain conditions (which are satisfied in our case) one can prove the following theorem (see \cite{HaLaSi}). 
\begin{thm}\label{thm:twistprod}The $A$-module $A\cdot \Delta$ can be endowed with a $\oo$-linear product 
$$\llangle a\cdot\Delta, b\cdot\Delta\rrangle:=\sigma(a)\cdot\Delta(b\cdot\Delta)-\sigma(b)\cdot\Delta(a\cdot\Delta),\quad a,b\in\AA$$satisfying
\begin{itemize}
\item[(i)] $\llangle a\cdot\Delta, b\cdot\Delta\rrangle=
  (\sigma(a)\Delta(b)-\sigma(b)\Delta(a))
  \cdot\Delta$;  
\item[(ii)] $\llangle a\cdot \Delta, a\cdot\Delta\rrangle=0$;
\item[(iii)] $\circlearrowleft_{a,b,c}\Big(\llangle \sigma(a)\cdot
  \Delta,\llangle
  b\cdot\Delta,c\cdot\Delta\rrangle\rrangle+ q\cdot\llangle
  a\cdot\Delta_,\llangle
  b\cdot\Delta,c\cdot\Delta\rrangle\rrangle\Big)=0$,
\end{itemize} where $a,b,c\in\AA$.
The product is in particular non-associative.
\end{thm}
\begin{dfn}\label{dfn:globhom}A \emph{hom-Lie
  algebra} is an $\AA$-module $\MM$ together
  with a $\oo$-bilinear product $\llangle\,\cdot,\cdot\,\rrangle$ on
  $\MM$ such that 
  \begin{description} 
       \item[(hL1.)] $\llangle a,a\rrangle =0$, for all $a\in \MM$;
       \item[(hL2.)] there is a $q\in \AA$ such that the identity $$\circlearrowleft_{a,b,c}\Big \{\llangle \sigma(a),\llangle
       b,c\rrangle\rrangle+q\cdot\llangle a,\llangle
       b,c\rrangle\rrangle\Big\}=0,$$ holds for all $a, b, c\in\MM$.
  \end{description}
\end{dfn}

If we let $\sigma$ vary in a group $G$, we get a collection of hom-Lie algebras, called an \emph{equivariant hom-Lie algebra}. Notice that an equivariant hom-Lie algebra includes a Lie algebra (which might certainly be abelian) since this is what we get when $\sigma$ is the unit element in $G$. We can thus view the equivariant structure as a deformation of a Lie algebra, parametrised by the elements in $G$. 

\begin{example}The obvious example is the left $\AA$-module $\AA\cdot \Delta$ with product defined by theorem \ref{thm:twistprod}. 
\end{example}

Some reasons why $\sigma$-derivations are important (and actually prevalent in abundance) in arithmetic and geometry, can be found in \cite{LarArithom} and the references therein. In short, $\sigma$-derivations and their associated hom-Lie algebras encodes the Galois structure (or more generally, the automorphism-structure) of the underlying algebra. 

\subsection{Infinitesimal hom-Lie algebras}\label{sec:infinitesimal_kummer}
We give a brief illustration of the above constructions that will be of some importance later. 

Let $\oo$ be a commutative algebra and put $\oo^!:=\oo[t]/t^n$, for some $n\geq 2$. Then $\oo^!$ is an \emph{infinitesimal thickening} of $\oo$. Put $\Delta:=a(\id-\sigma)$, where $a\in \oo$ and where $\sigma(t)=qt$, for some $q\in\oo^\times$. Then $(\oo^!\cdot \Delta, \,\llangle\,,\,\rrangle)$ is \emph{infinitesimal hom-Lie algebra on $\oo^!$}. The elements $\eps_i:=t^i\Delta$, $0\leq i\leq n-1$, forms a basis for $\oo^!\cdot\Delta$ as an $\oo^!$-module.

Now, put $\Delta:=a(\id-\sigma).$ Then $\Delta(t^i)=a(1-q^i)t^i$. Using Theorem \ref{thm:twistprod} (i) a small computation gives that 
$$\llangle \eps_i,\eps_j\rrangle=a(q^i-q^j)\eps_{i+j}.$$Observe that when $i+j\geq n$, then $\llangle \eps_i,\eps_j\rrangle=0$ since in that case $t^{i+j}=0$. In addition, $\llangle\eps_0,\eps_i\rrangle = a(1-q^i)\eps_i$, for all $i$. 
\begin{example}\label{exam:inf_n=3}When $n=3$, we get the $(a, q)$-parametric family of solvable $\oo$-Lie algebras:
$$\llangle \eps_0,\eps_1\rrangle =a(1-q)\eps_1,\qquad \llangle \eps_0,\eps_2\rrangle =a(1-q^2)\eps_2,\qquad \llangle \eps_1,\eps_2\rrangle=0.$$ One would be tempted to conjecture that infinitesimal hom-Lie algebras are Lie algebras for all $n$. However, this is not true in general.
\end{example}

Associated to $(\oo^!\cdot \Delta, \,\llangle\,,\,\rrangle)$ is an associative ``enveloping algebra'', defined by 

$$\mathbf{I}_q:=\frac{\oo\{\eps_0,\eps_1,\dots,\eps_{n-1}\}}{\Big(\eps_i\eps_j-q^{j-i}\eps_j\eps_i-(1-q^{j-i})\eps_{i+j}\Big)}.$$
See \cite{LarssonKummerWitt} for its construction.
\begin{example}For the algebra in example \ref{exam:inf_n=3} we find (after change of basis $\eps_0\mapsto \eps_0+1$) we get the algebra
\begin{equation}\label{eq:Aq3_inf}
 \mathbf{Q}^3_{\oo, q}:=\frac{\oo\{\eps_0,\eps_1,\eps_2\}}{\Big(
 \eps_0\eps_1-q\eps_1\eps_0,\,\,
  \eps_{0}\eps_2-q^2\eps_2\eps_{0},\,\,
    \eps_{1}\eps_2-q\eps_2\eps_{1}
\Big)}.
\end{equation}The space $\Mod\big(\mathbf{Q}^3_{\oo,q}\big)$ is a \emph{quantum affine three-space}. 
\end{example}

\subsection{Kummer--Witt hom-Lie algebras and their ``enveloping algebras''}\label{sec:Kummer_Witt}
 We keep the notation from above and further denote the algebra structure on $\AA$ by 
$$y_iy_j=\sum_{k=0}^n a^k_{ij}y_k, \quad a_{ij}^k\in \oo,$$ where the $y_i$ are the algebra generators of $\AA$. Let $\sigma$ be the $\oo$-linear algebra morphism on $\AA$ defined by $\sigma(y_i)=q_iy_i$, $q_i\in \oo^\times$ and let $\Delta=a(\id-\sigma)$ be the $\sigma$-derivation from the previous section. Put $\eps_i:=y_i\cdot\Delta$. Then, from Theorem \ref{thm:twistprod}, the pair 
\begin{equation}\label{eq:Witthom}
\W^\sigma_\AA:=(\AA\cdot \Delta, \llangle\,,\,\rrangle), \qquad \llangle \eps_i,\eps_j\rrangle=a\sum_{k=0}^n(q_i-q_j)a^k_{ij}\eps_k
\end{equation} defines a hom-Lie algebra structure on $\AA\cdot\Delta$. We call $\W^\sigma_\AA$ the \emph{Witt hom-Lie algebra} over $X$ attached to $\AA$ and $\sigma$. 
\begin{remark}The construction just given is obviously not dependent on the particular choice $\sigma(y_i)=q_iy_i$. Any other automorphism can be used. However, the result will, of course, be more complicated and harder to write out.
\end{remark}

From now on we assume that the $n$-th roots of unity are included in $\oo$. Fix a primitive such root $\zeta=\zeta_n$ and let $\AA$ be the Kummer extension
$$\AA=\oo[\sqrt[n]{x}]=\oo[t]/(t^n-x), \quad x\in \oo.$$ Fix an $r\in\Z$ and let $\sigma$ be the automorphism $\sigma(t)=\zeta^r t$.

The hom-Lie algebra of $\sigma$-twisted derivations on $A$ is called the \emph{Kummer--Witt hom-Lie algebra of level $r$} (see \cite{LarssonKummerWitt}) and its enveloping algebra is the \emph{Kummer--Witt algebra of level $r$}:
\begin{equation}\label{eq:kum_def} 
\eKW_x(r)=\frac{\oo\langle\eps_0,\eps_1,\dots,\eps_{n-1}\rangle}{\left(\eps_i\eps_j-\zeta^{r(j-i)}\eps_j\eps_i-(1-\zeta^{r(j-i)})x^\circlearrowright\,\,\eps_{\{i+j\,\,(\mathrm{mod} n)\}}\right)}.
\end{equation}The notation $x^\circlearrowright$ means that $x$ is included when $i+j\geq n$.

\begin{remark}The same game can clearly be played with Artin--Schreier extensions. We invite the reader to write out the corresponding relations for him/herself. 
\end{remark}
\begin{remark}\label{rem:ramification}
The above construction will be used with the following remark.  The element $x\in \oo$ is the \emph{$A$-ramification divisor}. This element determines a canonical subscheme in the ramification locus of a non-commutative space attached to $A$. The element $x$ is a ramification invariant in two senses: (1) as the divisor in $\oo$ over which $A$ is ramified (i.e., $\pp\mid x\Rightarrow \pp$ ramified); and (2) as an element giving a subscheme of the ramification locus in a certain non-commutative space (which we will construct later). 
\end{remark}

\subsubsection{Fibres of $\eKW_x(r)$}
Let $\pp\in\Spec(\oo)$ be a prime. Observe that the reduction of $\zeta$ modulo $\pp$ (i.e., the image of $\zeta$ in $k(\pp)$), $\bar\zeta$, is non-zero, since $\zeta\in \oo^\times$. Then 
\begin{equation}\label{eq:reductionKW}
\eKW_x(r)_{/\pp}=\frac{k(\pp)\langle\eps_0,\eps_1,\dots,\eps_{n-1}\rangle}{\left(\eps_i\eps_j-\bar\zeta^{r(j-i)}\eps_j\eps_i-(1-\bar\zeta^{r(j-i)})x^\circlearrowright\,\,\eps_{\{i+j\,\,(\mathrm{mod} n)\}}\right)}.
\end{equation}
We record the following for easy reference. 
\begin{prop}\label{prop:reductionKW}
There are the following three possibilities when reducing modulo a prime $\pp$:
\begin{itemize}
	\item[(1)] $\bar\zeta=1$, in which case we get
	$$\eKW_x(r)_{/\pp}=\frac{k(\pp)\langle\eps_0,\eps_1,\dots,\eps_{n-1}\rangle}{\left(\eps_i\eps_j-\eps_j\eps_i\right)}=k(\pp)[\eps_0,\eps_1,\dots,\eps_{n-1}],$$the commutative polynomial algebra;
	\item[(2)] $\bar x =0$, in which case we get
		$$\eKW_x(r)_{/\pp}=\frac{k(\pp)\langle\eps_0,\eps_1,\dots,\eps_{n-1}\rangle}{\left(\eps_i\eps_j-\bar\zeta^{r(j-i)}\eps_j\eps_i\right)},$$a \emph{quantum affine space};
	\item[(3)] and the \emph{generic case} (\ref{eq:reductionKW}) with relations unchanged.
\end{itemize}
It is important to note that in all three cases, the reduced algebra is a domain. 
\end{prop}
\begin{proof}
Obvious. 
\end{proof}

Some ring-theoretic properties of $\eKW_x(r)$ are summarised in the following theorem. Here will use notions that will not be defined in this paper. For more details see \cite{LarssonKummerWitt}. It is not essential to understand all words in the theorem in order to appreciate the rest of the paper. 

\begin{thm}\label{thm:KW}The algebra $\eKW_x(r)$ satisfies the following:
\begin{itemize}
	\item[(i)] it is an Auslander-regular, noetherian PI-domain, finite over its centre;
	\item[(ii)] $$\mathrm{Kdim}(\eKW_x(r))=\mathrm{gl.dim}(\eKW_x(r))= n+\mathrm{Kdim}(\oo),$$ where $\mathrm{Kdim}$ is the Krull dimension and $\mathrm{gl.dim}$ the global dimension;
	\item[(iii)] it is fibre-wise Cohen--Macaulay with $$\mathrm{GKdim}\big(\eKW_x(r)_{/\pp}\big)=\mathrm{tr.deg}\big(\eKW_x(r)_{/\pp}\big)=n,$$where $\mathrm{GKdim}$ denotes the Gelfand--Kirillov dimension;
	\item[(iv)] it is fibre-wise a maximal order in its fibre-wise division rings of fractions;
	\end{itemize}
\end{thm}
\begin{remark}The Auslander-regularity is a natural generalisations of regularity for non-commutative algebras. The same applies to Cohen--Macaulay-ness.  
\end{remark}

\section{Non-commutative algebraic spaces}\label{sec:nc_space}
In this section we will very briefly introduce the concept of non-commutative algebraic spaces in the sense of \cite{LarssonAritGeoLargeCentre}. We refer to this paper for all details. 

Let $k$ be a field, $A$ be a $k$-algebra and $\fam{M}:=\{M_1, M_2, \dots, M_s\}$ be a finite family of finite-dimensional (over $k$) $A$-modules with structure morphisms $\xi_i$. Then there is a non-commutative ring $\hat\OO_\fam{M}$ and a map $\boldsymbol{\xi}:A\to \hat\OO_\fam{M}$, encoding the \emph{simultaneous} deformations of the modules $M_i$. In fact, $\boldsymbol{\xi}$ is a versal family of $\fam{M}$. 

The ring object $\hat\OO_\fam{M}$ is constructed as follows (details can be found in \cite{EriksenLaudalSiqveland} or \cite{LarssonAritGeoLargeCentre}). Associated to $\fam{M}$ is a deformation functor having a matric pro-representing hull $(\hat H_{ij})$. The diagonal consists of quotients of non-commutative formal power series rings, and the entries off-diagonal are bimodules over the diagonal. Then $\OO_\fam{M}$ is defined as the matrix ring
$$\hat\OO_{\fam{M}}:=\big(\Hom_k(M_i, M_j)\otimes_k \hat H_{ij}\big).$$  The ring $\hat\OO_\fam{M}$ is to be thought of as the completion of the local ring at $\fam{M}$ in the moduli space $\Mod(A)$ of $A$-modules up to isomorphism. In all cases that I'm aware of this has a natural algebraisation $\OO_\fam{M}$, which should be viewed as the local ring itself. 

Taking the projective limit over all finite families $\fam{M}$, we define a ring object $\boldsymbol{\OO}$ on $\Mod(A)$. 
\begin{dfn}
 The \emph{non-commutative algebraic space} of $A$, is the pair
$$\Xscr_A:=\big(\Mod(A), \boldsymbol{\OO}\big),$$ where $\boldsymbol{\OO}$ is the \emph{structure object} of $\Xscr_A$. The \emph{tangent space} of $\Xscr_A$ at $\fam{M}$ is the matrix $T_\fam{M}:=(\Ext_A^1(M_i, M_j)_{ij})$. We say that $\Xscr_A$ has property $\mathbf{P}$ if $\Spec(\cent(A))$ has this property. 
\end{dfn}

\begin{remark}\begin{itemize}
	\item[(i)] Note that $\boldsymbol{\OO}$ is not a sheaf.
	\item[(ii)] Since $\OO$ is constructed by deformation theory, the structure object is only defined fibre-by-fibre over $\oo$. The construction \emph{can} be made in a formal sense without referring to any special type of base but this is not very useful in practice.  In fact:
	\item[(iii)] The ring $(\hat H_{ij})$ is constructed in the same way as in the commutative case (although it is more complicated): successively lifting the tangent structure (defined over a field) and killing obstructions to further liftings. 
\end{itemize}
\end{remark} 

The following is definition 2.8 in \cite{LarssonAritGeoLargeCentre}.
\begin{dfn}Let $k$ be a field and let $A$ and $S$ be $k$-algebras.
\begin{itemize}
	\item[(a)] An \emph{\'etale $S$-rational point} on $\Xscr_A$ is an algebra morphism
$$\xibf: A\to S$$ such that 
$$A/\ker\xibf=A/\mm_1\times A/\mm_2\times\cdots\times A/\mm_s$$ is a direct product of prime rings. If the $A/\mm_i$ are artinian, being prime is equivalent to being simple so the $\mm_i$ are maximal ideals. This applies in particular to the case when $S$ is artinian. 
	\item[(b)] The \emph{underlying \'etale point} of $\xibf$ is the set of algebras $\bar\xi_i:=A/\mm_i$. 
	\item[(c)] If all $\bar\xi_i$ are simple, the point is \emph{closed}, otherwise it is \emph{non-closed}. The point is \emph{open} if all the $\bar\xi_i$ are non-simple. 
	\item[(d)] If $s=1$, the map $\xibf$ is an \emph{$S$-rational point}, which we write in non-boldface: $\xi$; the algebra $\bar\xi=A/\mm$ is then the (unique) underlying point of $\xi$.  
	\item[(e)] If $S=\End_L(M)$ for some field $L$ and $M$ finite-dimensional over $L$, we say that $\xi$ is an \emph{$L$-rational point}.
	\item[(f)] $\xibf$ is a \emph{geometric \'etale point} if it is closed and $\cent(\bar\xi_i)=k^\mathrm{al}$ for all $i$.  
\end{itemize}
As usual we denote the $S$-rational points on $\Xscr=\underline{\Mod}(A)$ as $$\Xscr(S)=\underline{\Mod}(A)(S).$$ 
\end{dfn}

\section{The Jackson space}\label{sec:Jackson_space}

\subsection{Polynomial identity algebras}
Let $R$ be a commutative ring and let $ \Z\langle \mathbf{x}\rangle = \Z\langle x_1, x_2, \dots, x_n\rangle$ be the free $\Z$-algebra on $n$ generators. An $R$-algebra $A$ is a \emph{polynomial identity algebra} (\emph{PI-algebra}) if there is an $n$ and a $P(\mathbf{x})\in\Z\langle \mathbf{x}\rangle$ such that $P(a_1, a_2, \dots, a_n)=0$ for all $n$-tuples $(a_1, a_2, \dots, a_n)\in A^n$. 

Examples abound: all commutative algebras, Azumaya algebras (and so matrix algebras and central simple algebras) and algebras that are finite modules over their centres, are arguably the most important ones. 

Let $A$ be a PI-algebra with centre $\cent(A)$. Then every ideal in $A$ intersects $\cent(A)$ non-trivially. Furthermore, there is a dense open subscheme $\azu(A)\subset\Specm(\cent(A))$ such that for all $\mm\in\azu(A)$ the extension of $\mm$ to $A$ is also a maximal ideal. The subscheme $\azu(A)$ is called the \emph{Azumaya locus} of $A$. The complement, $$\ram(A):=\Specm(\cent(A))\setminus \azu(A),$$ is the support of a Cartier divisor and is called the \emph{ramification locus}. We will be a bit sloppy and refer to $\azu(A)$ and $\ram(A)$ both as subsets of $\cent(A)$ and as subschemes of $\Spec(A)$. 

We will routinely identify an $A$-module $M$ with its annihilator $\mathfrak{a}:=\Ann_A(M)$. If $M$ is a simple $A$-module, its annihilator is a primitive ideal. In the case of PI-algebras, an ideal being primitive is equivalent to it being maximal so we identify simple modules and maximal ideals via this correspondence.  

The following theorem is central in the theory of algebras finite over their centres. 
\begin{thm}[M\"uller's theorem]
Let $A$ be an affine PI-algebra over a field $k$ and let $M$ and $N$ be simple finite-dimensional $A$-modules. Then 
$$\Ext^1_A(M,N)\neq \emptyset\quad \iff\quad \Ann(M)\cap \cent(A)=\Ann(N)\cap \cent(A).$$(The annihilators are \emph{left} ideals.) The statement can be rephrased as the equivalence 
$$\Ext^1_A(M,N)\neq \emptyset\quad \iff\quad \Ann(M)\cap \cent(A)=\Ann(N)\cap \cent(A)\in\ram(A),$$since $\Ann(M)$, $\Ann(N)$ are maximal, and so also their intersections with the centre.
\end{thm}
Every simple module of a PI-algebra over a field is finite-dimensional over the ground field. 

We denote by $\Psi$ the map $\Spec(A)\to \Spec(\cent(A))$ defined by ideal contraction. When $A$ is finite over $\cent(A)$ this is a finite morphism in the sense that $\#\Psi^{-1}(\pp)<\infty$ for all $\pp\in\Spec(\cent(A))$.

\subsection{The Jackson algebra}
Now, to summarise the constructions made so far bullet-wise:
\begin{itemize}
	\item We have taken a Kummer extension $A=\oo[\sqrt[n]{x}]$, $x\in\oo$, of degree $n$;
	\item On this extension we have chosen a $\zeta^r_n$-derivation;
	\item From $\zeta_n^r$-derivation we constructed the associated hom-Lie algebra (which is an infinitesimal structure for twisted derivations, analogously to Lie algebras for ordinary derivations) and its ``enveloping algebra'', $\eKW_x(r)$. 
\end{itemize}Hence, $\eKW_x(r)$ captures the $\zeta_n^r$-infinitesimal structure of $\oo[\sqrt[n]{x}]$. 

Clearly, $\eKW_x(r)$  is quite complicated in general and so is not so easy to work with. Fortunately there is a subalgebra inside $\eKW_x(r)$ that has some remarkable ring-theoretic, geometric and arithmetic properties, some of which $\eKW_x(r)$ itself does not enjoy. This subalgebra is isomorphic to a ``$\zeta_n^r$-deformation'' of the enveloping algebra of the Lie algebra $\mathfrak{sl}_2$. We will now describe this algebra.  

Put, in order to simplify notation a bit, $\omega:=(1-\zeta^{2r})$. Recall that $\zeta=\zeta_n$ is assumed to be primitive. Since $\zeta^{-(n-1)}=\zeta$ and $\zeta^{-(n-2)}=\zeta^2$, the algebra
$$S:=\oo\langle \eps_0,\eps_1,\eps_{n-1}\rangle\big/I,$$ with
\begin{multline*}
I:=\Big(\eps_0\eps_1-\zeta^r\eps_1\eps_0- (1-\zeta^r)\eps_1, \,\,
  \eps_{n-1}\eps_0-\zeta^r\eps_0\eps_{n-1}- (1-\zeta^r)\eps_{n-1},\\
  \eps_{n-1}\eps_1-\zeta^{2r}\eps_1\eps_{n-1}- x\omega\eps_0\Big)
\end{multline*}is easily seen to be a subalgebra of $\eKW_x(r)$.
Now, define the algebra
\begin{equation}\label{eq:J}
	\eJac_x(r):=\oo\langle\eps_0,\eps_1,\eps_2\rangle/J, \qquad r\in\Z,
\end{equation}where $J$ is the ideal 
$$ J :=\Big(\eps_0\eps_1-\zeta^r\eps_1\eps_0,\quad
  \eps_2\eps_0-\zeta^r\eps_0\eps_2,\quad
  \eps_2\eps_1-\zeta^{2r}\eps_1\eps_2- x\eps_0-x\omega\Big).$$
\begin{prop}The algebras $S$ and $\eJac_x(r)$ are isomorphic over $\oo\left[\omega^{-1}\right]$, and in particular fibre-wise. Furthermore, $\eJac_x(r)$ (and hence also $S$) is an iterated Ore extension (skew polynomial ring). The generator $\eps_{n-1}$ in $S$ becomes the generator $\eps_2$ in $\eJac_x(r)$.
\end{prop}If $n$ is composite with at least two different prime factors, such that $\zeta^{2r}$ is primitive, then $\omega$ is invertible in $\oo$. Hence, in that case, $S$ and $\eJac_x(r)$ are isomorphic already over $\oo$.  
\begin{dfn}
The algebra $\eJac_x(r)$ is called the \emph{Jackson algebra of level $r$}. Let $\mathsf{w}$ be the Teichm\"uller character defined by $\sigma(\zeta)=\zeta^{\mathsf{w}(\sigma)}$. The family of algebras
$$\eJac_x^\mathrm{tot}:=\Big\{\eJac_x(\mathsf{w}(\sigma))\,\,\big\vert\,\, \sigma\in\Z/n\Big\},$$ is called the \emph{total Jackson algebra}. This is a $\Z/n$-torsor.
\end{dfn} 
We have already indicated the following remark, but it is important enough to repeat. 
\begin{remark}
The algebras $\eJac_x$, $\eJac_x^\mathrm{tot}$ and $\eKW_x(r)$, encode the Galois structure of $\oo[\sqrt[n]{x}]$ through its $\zeta^r$-derivations, even though we have reduced the number of generators from $n$ to $3$. 
\end{remark}

\begin{remark}\label{rem:J} 
\begin{itemize}
	\item[(i)]  The name ``Jackson algebra'' is in honour of F.H. Jackson (1870--1960) and his work on $q$-derivations. Indeed, $\eJac_x(r)$ is isomorphic to the (enveloping algebra of the) ``Jackson-$\mathfrak{sl}_2$'', a $q$-deformed version of the three-dimensional Lie algebra $\mathfrak{sl}_2$, deformed using $q$-derivations. See \cite{LaSi} for the details. 
	\item[(ii)] When $\zeta^{2r}=1$, the algebra $\eJac_x(r)$ is either isomorphic to the commutative polynomial algebra $\oo[t_1,t_2,t_3]$ (when $\zeta=1$) or to the $\oo$-algebra on generators $\eps_0$, $\eps_1$, $\eps_2$ and with relations
$$\eps_0\eps_1+\eps_1\eps_0=0,\quad \eps_2\eps_0+\eps_0\eps_2=0,\quad \eps_2\eps_1-\eps_1\eps_2=x\eps_0.$$
\end{itemize}
\end{remark} 

\subsection{The centre $\cent(\eJac_x(r))$}

To simplify notation we will from now on often omit the dependence on $r$ in the notation $\eJac_x(r)$ and simply write $\eJac_x$.

As before, we put
$$A_{/\pp}:=A\otimes_\oo k(\pp)$$for any $\oo$-algebra $A$. 

See \cite{LarssonKummerWitt} for the proofs of the following theorem and its corollary.
\begin{thm}\label{prop:centre_J}
Assume that $\zeta^r$ is a primitive $n$-th root of unity. Then the following holds. 
\begin{itemize}
	\item[(i)] For $x=0$:
	$$\cent(\eJac_0)_{/\pp}=k(\pp)\Big[\eps_0^a\eps_1^b\eps_2^c\,\,\,\Big\vert\,\,\, r(a+2b)\text{ and } r(b-c)\equiv 0 \,\,\mathrm{mod}\, n\Big].$$Observe that $k(\pp)[\eps_0^l, \eps_1^l, \eps_2^l]\subseteq \cent(\eJac_0)_{/\pp}$. 
	\item[(ii)] For $x\neq 0$:
	$$k(\pp)[\eps_0^l, \eps_1^l, \eps_2^l]\subseteq\cent(\eJac_x)_{/\pp}$$where $l\in\N$ is such that $lr$ is the least multiple of $n$.
\end{itemize}
In both cases we have
\begin{itemize}
	\item[(iii)] $\eJac_x$ is an Auslander-regular, homologically homogeneous\footnote{This is another regularity property, which I won't define, of non-commutative rings, generalising the commutative counterpart.}, noetherian, fibre-wise Cohen--Macaulay domain, finite as a module over its centre and hence a polynomial identity ring (PI) of $\mathrm{pideg}(\eJac_x)=n$;
	\item[(iv)] $\eJac_x$ is a maximal order in its division ring of fractions;
	\item[(v)] $\Spec(\cent(\eJac_x))$ is a normal, irreducible scheme of dimension three. 
	\item[(vi)] $\Spec(\cent(\eJac_0))$ is in addition \emph{fibre-wise Cohen--Macaulay} in the commutative sense, i.e., 
	$$\Spec\big(\cent(\eJac_x)_{/\pp}\big)=\Spec\big(\cent(\eJac_0)\otimes_\oo k(\pp)\big)$$ is a Cohen--Macaulay scheme for all $\pp\in\Spec(\oo)$. 
\end{itemize}
\end{thm}
In fact, in \cite{LarssonKummerWitt}, the centre in case (ii) is computed explicitly, but we won't be needing this in this paper.

We immediately get the following corollary.
\begin{corollary}\label{cor:Jacks_hom}With the notation as above:
\begin{itemize}
	\item[(i)] The algebra $(\eJac_x)_{/\pp}$ is finite as a module over $k(\pp)\big[\eps_0^l, \eps_1^l, \eps_2^l\big]$.
	\item[(ii)] Hence, the ring extension $k(\pp)\big[\eps_0^l, \eps_1^l, \eps_2^l\big]\subseteq \cent(\eJac_x)$ is finite, and consequently the morphism 
$$\psi: \,\,\Spec\big(\cent(\eJac_x)\big)\to \mathbb{A}^3_{(l)}:=\Spec\Big(k(\pp)\big[\eps_0^l, \eps_1^l, \eps_2^l\big]\Big)$$ is finite as a morphism of schemes. We put $l$ in the notation to indicate that we have a weighted version of the affine three-space. 
	\item[(iii)] In the other direction, any maximal $\mm$ in $\cent(\eJac_x)$ splits into $i$ maximal ideals in $\eJac_x$, where $1\leq i\leq m$, and where $m$ is the rank of $\eJac_x$ as a module over $\cent(\eJac_x)$. In other words, 
		$$\Psi^{-1}(\mm)=\Big\{\mathfrak{m}_1, \mathfrak{m}_2,\dots,\mathfrak{m}_i\,\, \Big\vert \,\,\text{for some $1\leq i\leq m$}\Big\}.$$Recall that $\Psi$ is the contraction map on ideals. 
		\item[(iv)] Also, any maximal $\mm$ in $k(\pp)\big[\eps_0^l, \eps_1^l, \eps_2^l\big]$ splits into $j$ maximal ideals in $\eJac_x(r)$, where $1\leq j\leq n$, and where $n$ is the rank of $\eJac_x$ as a module over $k(\pp)\big[\eps_0^l, \eps_1^l, \eps_2^l\big]$. In other words, 
		$$\big(\psi\circ\Psi\big)^{-1}(\mm)=\Big\{\mathfrak{m}_1, \mathfrak{m}_2,\dots,\mathfrak{m}_j\,\, \Big\vert \,\,\text{for some $1\leq j\leq n$}\Big\}.$$
\end{itemize}   
\end{corollary}

I conjectured in \cite{LarssonKummerWitt} that $\Spec(\cent(\eJac_0))$ is singular for all $n$, and that it might also be rational. In fact, I could prove the following theorem. 
\begin{thm}\label{thm:rationalCM}Let $K$ be a field of characteristic zero and let $\eJac_x^\mathrm{a}$ be the base change $(\eJac_x)_{/K}\otimes_K K^\mathrm{a}$ of $(\eJac_x)_{/K}$ to the algebraic closure $K^\mathrm{al}$. Then 
\begin{itemize}
	\item[(i)] $\Spec\big(\cent(\eJac_x^\mathrm{a})\big)$ has rational singularities;
	\item[(ii)] both $\Spec\big(\cent(\eJac_x^\mathrm{a})\big)$ and $\Spec\big(\cent((\eJac_x)_{/K})\big)$ are Cohen--Macaulay.
\end{itemize} 
\end{thm}
The question I couldn't, and still can't, answer in \cite{LarssonKummerWitt} is whether this also implies that $\Spec\big(\cent((\eJac_x)_{/K})\big)$ itself has rational singularities. 

\subsection{The non-commutative space $\Xscr_{\eJac_x}$}\label{sec:J}
\subsubsection{Properties of $\Xscr_{\eJac_x}$}
Since $\eJac_x$ is finite as a module over its centre, we can directly infer the following result from the previous discussions. 

\begin{corollary}The space  $\Xscr_{\eJac_x}$ satisfies the following properties.
\begin{itemize}
	\item[(i)] It is Auslander-regular, homologically homogeneous, noetherian and irreducible.
	\item[(ii)] It is fibre-wise Cohen--Macaulay of dimension 3.
	\item[(iii)] It is fibre-wise \emph{rational} in the sense that it is a maximal order in its division ring of fractions. 
	\item[(iv)] It has fibre-wise (at worst) \emph{geometric} rational singularities in that its base change to the algebraic closure has rational singularities.  
	\item[(v)] Contraction of ideals defines a finite morphism  $\Xscr_{\eJac_x}\to \Spec(\cent(\eJac_x))$ which is one-to-one over $\azu(A)$, and $j$-to-one, for some $1\leq j\leq n$, over $\ram(A)$. We view  $\Xscr_{\eJac_x}$ as fibred over $\Spec(\cent(A))$. 
\end{itemize}   
\end{corollary}
Note that $\Xscr_{\eJac_x}$ satisfies regularity properties generalising the commutative regularity condition even though the centre might have singularities.  

We put
$$\Xscr_{\eJac_x}^\mathrm{tot}:=\bigsqcup_{\sigma\in\Z/n} \Xscr_{\eJac_x(\mathsf{w}(\sigma))},$$the \emph{total Jackson space}. This is a $\Z/n$-torsor via the $\Z/n$-action on $\eJac_x^\mathrm{tot}$. We can partition $\Xscr_{\eJac_x}^\mathrm{tot}$ with respect to subgroups of $\Z/n$ into sets (not disjoint of course).  

\subsubsection{The locus of one-dimensional points and its infinitesimal structure}\label{sec:point_locus_tangents} 
For illustration we will now study the rational one-dimensional points. Thus, for this section we let $k$ be \emph{any} field with a primitive $n$-th root of unity $\zeta:=\zeta_n$. Furthermore, we identify $1$-dimensional points, $1$-dimensional modules and their annihilator ideals. These points are maximal since $1$-dimensional (over $k$) modules are simple. 

Note first that any $1$-dimensional $\eJac_x$-module must lie on the ramification locus. Indeed, any module in $\azu(A)$ corresponds to a central simple algebra of dimension $n^2$.

Assume $\zeta^{2r}\neq 1$. Then the relations reduce to 
\begin{align*}
(1-\zeta^r)\eps_0\eps_1&=0\\
(1-\zeta^r)\eps_0\eps_2&=0\\
(1-\zeta^{2r})\eps_1\eps_2&=x\eps_0+x(1-\zeta^{2r}).
\end{align*}This implies that $x(\eps_0+(1-\zeta^{2r}))=0$ so, if $x\neq 0$, we have that $(\zeta^{2r}-1,0,0)$ is a one-dimensional module. In addition, the conic (hyperbola) $C_x$ given by $\eps_1\eps_2=x$ in the plane $\eps_0=0$ consists entirely of commutative points (i.e., one-dimensional modules). When $x=0$, we find that the one-dimensional modules lie on the union of the coordinate planes. 

\begin{remark}\label{rem:zeta2=1}When $\zeta^{2r}=1$ we may assume that $\zeta^r=-1$. The relations
$$\eps_0\eps_1+\eps_1\eps_0=0,\quad \eps_2\eps_0+\eps_0\eps_2=0,\quad \eps_2\eps_1-\eps_1\eps_2=x\eps_0$$imply that
$$2\eps_0\eps_1=0,\quad 2\eps_0\eps_2=0,\quad x\eps_0=0.$$Consider first when $\mathrm{char}(k)\neq 2$. When $x\neq 0$ we get that the plane $\{\eps_0=0\}$ consists entirely of one-dimensional points. On the other hand, when $x=0$, we get the union of the coordinate planes. 
Now, if $\mathrm{char}(k)=2$, we get, when $x\neq 0$, the plane $\{\eps_0=0\}$ and, if $x=0$, the whole $\mathbb{A}^3$. 
\end{remark}
%
%

We begin by working in $\eps_0=0$. Let $k\subseteq k'$ be a field extension and consider the $k'$-rational point 
$$\xi_{(a, b)}:= k'\cdot\ff, \quad\text{with}\quad \eps_1\cdot\ff=a\ff,\quad \eps_2\cdot\ff=b\ff, \quad a,b\in k',$$and similarly $\xi_{(u, v)}$. Then we have the following results.
\begin{prop}Put $\xibf:=\big\{\xi_{(a, b)},\xi_{(u, v)}\big\}$ be an \'etale $k'$-rational point on $\Xscr_{\eJac_x}$. Recall that $T_{\xibf}$ denotes the tangent space at $\xibf$.
\begin{itemize}
	\item[(i)] Suppose $x\neq 0$ and $\zeta^{2r}\neq 1$. Then, 
$$
	T_{\xibf}=\Ext^1_{\eJac_x}\!\!\big(\xi_{(a, b)},\xi_{(u, v)}\big)=\begin{cases} k,& \text{when $a=\zeta^r u$, and $v=\zeta^r b$}  \\
	k', & \text{when $a=\zeta^{2r}u$, and $v=\zeta^{2r}b$} \\
	k', & \text{when $a=u$, and $b=v$}.
	\end{cases}
$$	
In all other cases $\Ext^1_{\eJac_x}\!\!\big(\xi_{(a, b)},\xi_{(u, v)}\big)=0$. 
\item[(ii)] Suppose $x=0$ and $\zeta^{2r}\neq 1$. Then
$$
T_{\xibf}=\Ext^1_{\eJac_x}\!\!\big(\xi_{(a, b)},\xi_{(u, v)}\big)=\begin{cases}
k', & \text{when $a=u$, and $b=v$}\\
k', & \text{when $a=\zeta^{2r} u$, and $v=\zeta^{2r} b$}\\
(k')^2, & \text{if when $a=b=u=v=0$}\\
0, & \text{otherwise.}
\end{cases}
$$
\item[(iii)] Suppose $x\neq 0$ and $\zeta^{2r}= 1$. Then
$$
	T_{\xibf}=\Ext^1_{\eJac_x}\!\!\big(\xi_{(a, b)},\xi_{(u, v)}\big)=\begin{cases}
	k', & \text{when $a=u$, and $b=v$}\\
	0, & \text{otherwise.}
	\end{cases}
$$	
\item[(iv)] Suppose $x=0$ and $\zeta^{2r}=1$. Then
$$
	T_{\xibf}=\Ext^1_{\eJac_x}\!\!\big(\xi_{(a, b)},\xi_{(u, v)}\big)=\begin{cases}
	k', & \text{when $a=\pm u$, $b=\pm v$, $\mathrm{char}(k)\neq 2$ }\\
	(k')^3 & \text{when $a=u$, $b=v$, $\mathrm{char}(k)= 2$. }
	\end{cases}
$$	
\end{itemize}
\end{prop} 
\begin{proof}
The proof (which is rather simple) can be found, together with other information, in \cite{LarssonKummerWitt}. 
\end{proof}

From the above, one can construct the ring object $\boldsymbol{\OO}$ over families of $1$-dimensional modules. Even though the result will probably be hard to interpret I feel that it can be helpful to see what $\boldsymbol{\OO}$ actually look like in practice. 

We only write it out for families of two modules.

\begin{prop}Let the notation be as above. Modulo possible obstructions the versal bases for the deformations of $\xibf$ are the following. 
\begin{itemize}
	\item[(i)] Suppose $x\neq 0$ and $\zeta^{2r}\neq 1$. Then, for $j=1,2$, 
$$\hat\OO_{\xibf}=\begin{pmatrix}
	\End_{k'}(\xi_{(\zeta^{jr}u, b)})\otimes k'[[t_{11}]] & \!\!\!\!\!\!\!\!\Hom_{k'}(\xi_{(\zeta^{jr}u, b)},\xi_{(u, \zeta^{jr}b)})\otimes\langle t_{12}\rangle\\
	\Hom_{k'}(\xi_{(u, \zeta^{jr}b)},\xi_{(\zeta^{jr}u, b)})\otimes\langle t_{21}\rangle &\!\!\!\!\!\!\!\!\End_{k'}(\xi_{(\zeta^{jr}u, b)})\otimes {k'}[[t_{22}]]
\end{pmatrix},$$with algebraisation, 
$$\OO_{\xibf}=\begin{pmatrix}
	\End_{k'}(\xi_{(\zeta^{jr}u, b)})\otimes{k'}[t_{11}] & \!\!\!\!\!\!\!\!\Hom_{k'}(\xi_{(\zeta^{jr}u, b)},\xi_{(u, \zeta^{jr}b)})\otimes\langle t_{12}\rangle\\
	\Hom_{k'}(\xi_{(u, \zeta^{jr}b)},\xi_{(\zeta^{jr}u, b)})\otimes\langle t_{21}\rangle &\!\!\!\!\!\!\!\!\End_{k'}(\xi_{(\zeta^{jr}u, b)})\otimes {k'}[t_{22}]
\end{pmatrix},$$in all non-zero cases. In the other cases 
$$\hat\OO_{\xibf}=\OO_{\xibf}=\Hom_{k'}(\xi_{(a, b)},\xi_{(u, v)}).$$
\item[(ii)] Suppose $x=0$ and $\zeta^{2r}\neq 1$. Then, the first two rows combine to the case in (i), while the second is
$$\hat\OO_{\xi_{(0, 0)}}= \End_{k'}(\xi_{(0, 0)})\otimes {k'}\langle\!\langle t_1, t_2\rangle\!\rangle, \quad \OO_{\xi_{(0, 0)}}= \End_{k'}(\xi_{(0, 0)})\otimes {k'}\langle t_1, t_2\rangle.$$
\item[(iii)] Suppose $x\neq 0$ and $\zeta^{2r}= 1$. Then
$$
	\hat\OO_{\xi_{(a, b)}}=\End_{k'}(\xi_{(a,b)})\otimes {k'}[[t]],\quad \OO_{\xi_{(a, b)}}=\End_{k'}(\xi_{(a,b)})\otimes {k'}[t]
$$	and $\hat\OO_{\xibf}=\OO_{\xibf}=\Hom_{k'}(\xi_{(a, b)},\xi_{(u, v)})$ in all other cases. 
\item[(iv)] Suppose $x=0$ and $\zeta^{2r}=1$ and let $\xibf$ be the \'etale rational point $\xibf=\big\{\xi_{(a, b)}, \xi_{(-a,-b)}\big\}$. Then, when $\mathrm{char}(k)\neq 2$, 
$$
\hat\OO_{\xibf}=\begin{pmatrix}
	\End_{k'}(\xi_{(a, b)})\otimes {k'}[[t_{11}]] & \Hom_{k'}(\xi_{(a, b)}, \xi_{(-a, -b)})\otimes\langle t_{12}\rangle\\
	\Hom_{k'}(\xi_{(-a, -b)}, \xi_{(a, b)})\otimes \langle t_{21}\rangle &  \End_{k'}(\xi_{(-a, -b)})\otimes {k'}[[t_{22}]] 
\end{pmatrix}$$and when $\mathrm{char}(k)=2$, 
$$\hat\OO_{\xi_{(a,b)}}=\End_{k'}(\xi_{(a,b)})\otimes{k'}\langle\!\langle t_1, t_2, t_3\rangle\!\rangle,$$ both with obvious algebraisations. 
\end{itemize}Be sure to notice the situation when $\mathrm{char}(k)=2$. 
\end{prop} 
\section{Ramification}\label{sec:ram}

We will separate the discussion into two cases: (1) the geometric case, and (2) the arithmetic case. The geometric case is essentially trivial. 
\subsection{The geometric case}
This is the case when the base $\oo$ is an algebra over a field $k$, such that $\zeta_n\in k$. The discussion below applies to both $\eJac_x$ and $\eKW_x$, but we will only write it out for $\eJac_x$. 

Assume $\oo=k[b_1, b_2, \dots, b_m]/I$. Hence
$$A=\oo[\sqrt[n]{x}]=\oo[t]/(t^n-x)=k[b_1, b_2,\dots b_m, t]/(I, t^n-x), \quad x\in \oo.$$The zero-locus of $x$ is the branch locus, $\branch(A_{/\oo})$, of $A_{/\oo}$. Over $\pp\notin\branch(A_{/\oo})$ the cover $\Spec(A)\to \Spec(\oo)$ is \'etale (and in particular unramified) and so the fibre over $\pp$ is
$$A\otimes_\oo k(\pp) = \prod_{i=1}^\ell k_i, \quad 1\leq \ell\leq \deg(A_{/\oo}),$$ and where each $k_i$ a finite field extension of $k(p)$. The group $\Z/n$ acts transitively on the points of the fibre. Therefore we can construct the algebras $\eJac_x$ (and $\eKW_x$) fibre-wise, but that is not what we are going to do. 

Since $n$ is invertible on the base the only interesting case of ramification is when $\pp\mid x$, i.e., when $\pp\in\branch(A_{/\oo})$. Therefore, 
\begin{prop}Let $\oo$ and $A$ be as above. Then
\begin{itemize}
	\item[(i)] if $\pp\notin\branch(A_{/\oo})$, the reduction of $\eJac_x$ modulo $\pp$ gives a generic (i.e., $x\neq 0$) $\eJac_{x, /\pp}$;
	\item[(ii)] if $\pp\in\branch(A_{/\oo})$, the reduction of $\eJac_x$ modulo $\pp$ gives a quantum affine algebra $\eJac_{0, /\pp}$. 
\end{itemize}
Hence, the ramification properties of $A_{/\oo}$ determines the structure of the algebras $\eJac_x$ fibre-wise. The converse is also true in the geometric case: the structure of $\eJac_x$ can be used to read off the ramification properties of $A_{/\oo}$ fibre-wise.  
\end{prop}
\subsection{The arithmetic case}
In a sense the geometric case is simple enough, it only comes down to the fact that Kummer extensions have simple properties when $n$ is invertible on all fibres. 

The arithmetic case is not so straightforward since the fibres have characteristic that can divide $n$. Let $\oo_K$ be the ring of integers in a number field $K$ such that $\zeta_n\in \oo_K$. 

Consequently, $A=\oo_K[\sqrt[n]{x}]$,  $x\in\oo_K$.
The branch locus is still the zero-locus of $x$, but now there are more possible types of ramification. In this paper we will only consider ramification at closed points, so ramification coming from inseparability (or imperfectness) of the residue class rings, will not be an issue. 

Assume now that $p\mid n$, say $n=p^km$. Then, $\zeta^n=1\iff (\zeta_n^m-1)^{p^k}=0$, implying that $\zeta^m=1$ for $m<n$. Hence there are no, non-trivial, primitive $n$-th roots of unity in this case.

\begin{prop}\label{prop:arithmJxram}Let $A=\oo_K[\sqrt[n]{x}]$,  with $x\in\oo_K$, and let $\pp\in\Spec(\oo_K)$. 
\begin{itemize}
	\item[(i)] If $\pp\nmid n$ and $\pp\nmid x$, the extension is unramified and $\Xscr_{\eJac_x}\otimes k(\pp)$ is generic;
	\item[(ii)] if $\pp\nmid n$ and $\pp\mid x$, the extension is tamely ramified (total if $\pp^2\nmid x$) and $\Xscr_{\eJac_0}\otimes k(\pp)$ is a quantum affine space;
	\item[(iii)] if $\pp\mid n$ and $\pp\nmid x$, the extension is ramified and
	 $$\Xscr_{\eJac_x}\otimes k(\pp)=\Xscr_{\mathbf{A}^1_{k(\pp)}},$$where 
	 $$\mathbf{A}^1_{k(\pp)}:=\frac{k(\pp)[u]\langle t, \partial\rangle}{(\partial t-t\partial - u)},$$ is the first Weyl algebra over the affine line, which in characteristic $p$, is Azumaya, and, finally
	\item[(iv)] if $\pp\mid n$ and $\pp\mid x$, the extension is ramified and $\Xscr_{\eJac_x}\otimes k(\pp)$ is the affine three-space.
\end{itemize}
\end{prop}

Observe that we cannot distinguish the ramification types (tame/wild) in (iii) and (iv). On the other hand, the partition of $\Xscr_{\eJac_x}^\mathrm{tot}$ with respect to subgroups of $\Z/n$ mentioned before can clearly be used with the ramification groups. However, the result will obviously not be visible in the structure of the spaces $\Xscr_{\eJac_x}$ themselves. Nevertheless, maybe studying the relation (via non-commutative ``morphisms'' between the elements in the partitions) might be used in some way.  

\begin{proof}If $\pp\nmid n$ and $\pp\nmid x$ the extension is unramified and $\eJac_{x/\pp}$ is a generic Jackson algebra since all $n$-the roots of unity are in $k\big(\pp)$. On the other hand, if $\pp\nmid n$ and $\pp\mid x$, the ramification is tame and we end up in the affine quantum algebra case $\eJac_{0/\pp}$. If, in addition $\pp^2\nmid x$, the ramification is total. If $\pp\mid n$, then every $\zeta_n$ reduces to $1$ modulo $\pp$ so get two the cases in (iii) and (iv).
\end{proof}

\section{Torsion points and elliptic Jackson spaces }\label{sec:torsionelliptic}

We will now construct the algebra $\eJac_x$ from the torsion field of elliptic curves with a $K$-rational torsion point and show that $\eJac_x$ encodes the reduction properties of the elliptic curve. 

Let $K$  be a local field and adjoin $\zeta_p$, a primitive $p$-th root of unity for $p$ a prime and put $L:=K(\zeta_p)$. Let $E$ be an elliptic curve over $K$ with a (non-trivial) $K$-rational $p$-torsion point. Then $K(E[p])/L$ is a Kummer extension of degree $p$ or $1$, with $K(E[p])$ the field generated by the $p$-torsion points of $E$. Hence, $K(E[p])=L[\sqrt[p]{x}]=L[t]/(t^p-x)$, for some $x\in L$, or $K(E[p])=L$. See \cite{Yasuda_Kummer}, for instance. We can, and will, assume that $x\in \oo_{L}$. In order to achieve this we might need to multiply $x$ with a unit (clearing denominators).  

The criterion of N\'eron--Ogg--Shafarevich (NOS) gives that if $E$ has good reduction at $\ell\neq p$, $\ell\in\Spec(\Z)$, then $K(E[p])/K$ is unramified at $\ell$. We will need the following generalisation due to M. Kida. We will denote the reduction types of elliptic curves with the corresponding Kodaira symbols.
\begin{thm}\label{thm:Kida}Let $E$ be an elliptic curve over a local field $K$, finite over $\Q_\ell$. Put $n:=-v(j(E))$, where $v$ is the valuation of $K$.  
	\begin{itemize}
		\item[(a)] If $p\geq 3$, $K(E[p])/K$ is unramified at the place of $K$ over $\ell$, if and only if 
		$$\mathrm{Type}(E) = \begin{cases}
		I_0, &\text{or}\\
		I_n, &\text{if $p\mid n$}.
		\end{cases}$$In other words, $K(E[p])/K$ is unramified if and only if $E$ has semistable reduction (over $K$). 
		\item[(b)] If $p=2$, $K(E[2])/K$ is unramified if and only if 
		$$\mathrm{Type}(E) = \begin{cases}
		I_0, &\text{or}\\
		I_n, &\text{if $n$ even, or}\\ 
		I^\ast_0\text{ or } I^\ast_n, &\text{with $n$ odd.}
		\end{cases}$$
	\end{itemize}
\end{thm}
\begin{proof}
See \cite[Theorem 1.1]{Kida_Ramification_Division_fields}
\end{proof}
 
 The following lemma holds also when $K$ i a number field.
\begin{lem}\label{lem:Cohen_cyclotomic}
A prime $\pp$ in $K$ is ramified in $L=K(\zeta_p)$ if and only if $\pp\mid p$ and $(p-1)\nmid e(\pp)$, where $e(\pp)$ is the absolute ramification index of $\pp$. In particular, if $e(\pp)=1$ and $p\geq 3$, $\pp$ is ramified if and only if $\pp\mid p$. If $p=2$, $\pp$ is ramified if and only if $\pp\mid p$ (regardless of $e(\pp)$). 
\end{lem}
\begin{proof}
See \cite[Theorem 1.1(1)]{Cohen_et_al_Cyclotomic}.
\end{proof}

Put $F:=K(E[p])$ and consider the following tower of extensions:

$$\xymatrix{F & \ar@{..}[l]\qq_F\\
\ar@{-}@/^1pc/[u]^{\text{deg $\,=\, p$ or $1$}}L\ar@{-}[u] &\ar@{..}[l]\qq\ar@{-}[u]\\
\ar@{-}@/^1pc/[u]^{\text{deg} \,=\, p-1}K\ar@{-}[u] & \ar@{..}[l]\qq_K\ar@{-}[u]\\
\Q\ar@{-}[u]& \ar@{..}[l]\ell\ar@{-}[u]}$$Throughout, we put $k:=k(\qq)$, the residue class field of $\qq$. 

Observe that $[F:L]=p$ or $[F:L]=1$, $[L:K]=p-1$, so either $[F:K]=p(p-1)$ or $[F:K]=p-1$. 

\begin{remark}It seems that the case $[F:L]=1$ is rare. Recall that the roots of the modular equation $\Phi_p(X, j(E))=0$ are the $j$-values of all elliptic curves isogeneous to $E$ with cyclic kernel of degree $p$. From this follows that any cyclic isogeny is defined over a subfield of $K(E[p])$ and so
$$L\subseteq L\cdot K(\Phi_p(X, j(E)))\subseteq K(E[p]).$$ Therefore a necessary condition for $F=K(E[p])$ to be equal to $L$ is that the modular polynomial $\Phi_p(X, j(E))$ (of degree $p$) splits completely over $L$, i.e., that $L$ is a splitting field for $\Phi_p(X, j(E))$. Hence it is necessary that $E$ is represented by an $L$-rational point on the modular curve $X_0(p)_{/\Q}$. Further evidence of this scarcity can be found in \cite{Jimenez-Lozano}, where it is proved that if $K=\Q$, the only possible values for $n$ where $\Q(E[n])=\Q(\zeta_n)$ can happen are $2, 3, 4$ and $5$. 
\end{remark}
\begin{remark}Assuming the degree of $F/L$ is bigger than one, it is fairly simple to find a generator $x$ such that $F=L[\sqrt[p]{x}]$. See the (sketch of) the proof of \cite[Proposition 2.5]{Yasuda_Kummer}. 
\end{remark}

I want to point out that many of the statements made below are well-known. However, I made the decision to prove most of them anyway, probably with more difficult proofs than needed. The main reasons are that I had problem finding suitable references and also that I found some joy in coming up with the proofs, in addition to fresh insights.

We need to separate the following into two cases: $\ell \neq p$, and $\ell=p$. 
\subsection{Special fibres of elliptic Jackson spaces; the case $\ell \neq p$}\label{sec:l!=p}
By hypotheses here then, $\qq\nmid p$. Then $\qq_K$ is unramified in $L$ by lemma \ref{lem:Cohen_cyclotomic}. Assume $[F:L]\neq 1$. 
\begin{lem}\label{lem:qunramified}Let $\qq$ be a prime in $L$ over $\qq_K$. Then $\qq$ is unramified in $F/L$ if and only if $v_{\qq}(x)\equiv 0\,\, (\text{mod $p$})$. In particular, if $\qq\nmid x$, $\qq$ is unramified. 
\end{lem} 
\begin{proof}
	See \cite[Exercise 2.12]{Cassels_Frohlich}.
\end{proof}
\begin{remark}\label{rem:unramified}
Notice that $\qq$ can be unramified even if $\qq\mid x$. Namely if 
$$x=\mathfrak{a}\qq^{p^t}, \,\,\, t\geq 1\quad \text{or, equivalently,}\quad v_{\qq}(x)\not\equiv 0\,\, (\text{mod $p$}),$$ for  some ideal $\mathfrak{a}$ in $\oo_L$, relatively prime to $\qq$. This is seen by a change of divisor $x$ to an isomorphic, unramified extension (see \cite[Exercise 2.12]{Cassels_Frohlich}). In other words, we can change $x$ to an element $y\in L$ such that 
$$F= L[\sqrt[p]{x}]\simeq L[\sqrt[p]{y}], \quad \text{with $\qq\nmid y$}.$$
\end{remark}
Now since $\qq_K$ is unramified in $L$, $\qq$ is unramified in $F$ if and only if $v_{\qq}(x)\equiv 0\,\, (\text{mod $p$})$. Also, since $\ell\neq p$, $\zeta_p \not\equiv 1\,\, (\text{mod $\ell$})$. Note that
\begin{align}\label{eq:ram_index}
	e(\qq_F\vert \qq_K)=e(\qq_F\vert \qq)e(\qq\vert \qq_K).
	\end{align}

\begin{thm}\label{prop:semi}Suppose $\ell\neq p$. Then $E$ has semistable reduction at $\qq_K\mid \ell$ if and only if, possibly after a change of generator $x$,  $\eJac_x\otimes_{L} k$ is generic, or, equivalently, that $\Xscr_{\eJac_x}\otimes_L k$ is a generic Jackson space.  
\end{thm}
The claim is independent of $r$ since $\zeta^r$ is primitive for all $1\leq r\leq p-1$. However, the resulting spaces are non-isomorphic in general. 

\begin{proof}Assume first that $E$ is semistable. This is, by Kida's theorem \ref{thm:Kida}, equivalent to $\qq_K$ unramified in $F/K$. Since the left-hand-side of (\ref{eq:ram_index}) is then $1$, we see that $\qq$ is unramified in $F/L$. By lemma \ref{lem:qunramified} this is equivalent to $v_{\qq}(x)\equiv 0\,\,(\text{mod $p$})$. Therefore, either $\qq\nmid x$ and the reduction becomes generic; or $x=\mathfrak{a}\qq^s$, with $p\mid s$, in which case remark \ref{rem:unramified} implies that we can change generator $x\to y$ so that $F=L[\sqrt[p]{y}]$ and $\qq\nmid y$, and therefore the reduction is still generic (after change to $y$). 

Conversely, suppose $\eJac_x\otimes_{L}k$ is generic. Then $x\not\equiv 0\,\,(\text{mod $\qq$})$. Hence $v_{\qq}(x)\equiv 0\,\,(\text{mod $p$})$, implying that $\qq$ is unramified in $F$ by lemma \ref{lem:qunramified}. From (\ref{eq:ram_index}) we thus find that $e(\qq_F\vert \qq_K)=e(\qq\vert\qq_K)$. Since $\qq_K\nmid p$, lemma \ref{lem:Cohen_cyclotomic} implies that $e(\qq\vert \qq_K)=1$ and so $\qq_K$ is unramified in $F$. Kida's theorem then gives that $E$ is semistable. The proof is complete. 
\end{proof}

\begin{thm}\label{prop:nonsemi}
Suppose $\ell\neq p$. If $E$ has non-semistable reduction at $\qq_K\mid\ell$, then $\Xscr_{\eJac_x}\otimes_{L}k$ is a (possible trivial, depending on the roots of unity in $k$) quantum affine space. 

Suppose, conversely, that $\Xscr_{\eJac_x}\otimes_{L}k$ is a quantum affine space, implying that $x=\mathfrak{a}\qq^s$. 
\begin{itemize}
	\item[(i)] If $p\nmid s$, then $E$ has non-semistable reduction;
	\item[(ii)] if $p\mid s$, then $E$ has semistable reduction. 
\end{itemize}
\end{thm}
\begin{proof}
	Since $\qq_K\nmid p$, lemma \ref{lem:Cohen_cyclotomic} implies that $\qq_K$ is unramified in $L$. From (\ref{eq:ram_index})
	 we find that $e(\qq_F\vert \qq_K)=e(\qq_F\vert\qq_K)$ and so if $\qq$ ramifies in $F$, $\qq$ must also ramify in $F$. Therefore, lemma \ref{lem:qunramified} implies that $v_{\qq}(x)\not\equiv 0\,\,(\text{mod $p$})$. Hence $x=0$ in $\eJac_x\otimes_{L}k$.
	
	Conversely, assume $x=\mathfrak{a}\qq^s$. If $p\nmid s$ then lemma \ref{lem:qunramified} gives that $\qq$ is ramified in $F$, implying that $\qq_K$ is also ramified in $F$. Kida's theorem then implies that $E$ has non-semistable reduction. On the other hand, if $p\mid s$, then $\qq$ is unramified in $F$ (after a change of generator). Equation (\ref{eq:ram_index}) gives $e(\qq_F\vert \qq_K)=e(\qq\vert\qq_K)$, implying that $\qq$ is unramified in $F$, since $e(\qq\vert \qq_K)=1$ when $\qq_K\nmid p$ by lemma \ref{lem:Cohen_cyclotomic}. Therefore, Kida's theorem implies that $E$ has semistable reduction. 
\end{proof}
Clearly, we didn't need to prove theorem \ref{prop:nonsemi} since it is a consequence of theorem \ref{prop:semi}. However, I feel that the above proof gives me some interesting insights.
\begin{remark}It is interesting to note that, by \cite[Theorem IV.10.3(b)]{SilvermanAdvanced}, $E$ has good reduction over $K(E[\ell])$. 
\end{remark}
\subsection{Special fibres of elliptic Jackson spaces; the case $\ell = p$}\label{sec:l=p}

Since $\qq\mid p$, we have that $\zeta_p \equiv 1 \,\,(\text{mod $p$})$ unless $(p-1)\mid e(\qq)$. We assume from now on that $e(\qq)=1$, and so we can observe that $\bar\zeta_p =1\in k$. The case $e(\qq)\geq 2$ is probably considerably harder. As always, the case $p=2$, presents additional difficulties so we assume for simplicity that $p>2$. 


\begin{thm}\label{thm:reduction}
Let $E_{/K}$ be an elliptic curve with a $K$-rational $p$-torsion point. Assume that $[F:L]\geq 2$ (i.e., that not all $p$-torsion points are $K$-rational) and that $e(\qq)=1$. 
\begin{itemize}
\item[(i)] Let $E_{/K}$ be an elliptic curve with a $K$-rational $p$-torsion point and let $\qq_K$ be a prime of $K$ over $p$. Then $E$ cannot have good supersingular reduction at $\qq_K$; in addition,
\item[(ii)] the case of non-split multiplicative reduction cannot happen. 
\item[(iii)] If $E_{/K}$ has good ordinary, split multiplicative, or additive, reduction at $\qq_K$, then
$$\Xscr_{\eJac_x}\otimes_L k=\begin{cases}
\Xscr_{\mathbf{A}^1_{k}}, &\text{if $\qq\nmid x$}\\
\mathbb{A}^3_k, &\text{if $\qq\mid x$.}
\end{cases}$$
\end{itemize}
Recall that $\mathbf{A}^1_{k}$ denotes the first Weyl algebra over the affine $k$-line.
\end{thm}
Parts of the above theorem are almost certainly known but I haven't found precise references so I supply complete proofs of all statements. There are almost certainly easier and shorter arguments than presented.  
\begin{proof}
Assume that $E_{/K}$ has good supersingular reduction at $\qq_K\mid p$. Since $E$ has a $K$-rational $p$-torsion point, we know that $F/L$ is either trivial or cyclic of degree $p$. In the first case there is nothing to prove, so assume that the degree of the extension is $p$. This means that $[F:K]=p(p-1)$. By the assumption that $E$ has good supersingular reduction, \cite[Proposition 12, 1.11]{Serre_Galois_Elliptic} implies that the inertia group of $\mathrm{Gal}(F/K)$ has degree $p^2-1$, which is then clearly not possible for an extension of degree $p(p-1)$. 


Suppose that $E_{/K}$ has multiplicative reduction over $p$. Let $E_q$ be the associated Tate curve, with $v_\qq(q)<0$ (see \cite[Theorem V.5.3]{SilvermanAdvanced} for details). Then $E$ has a $K$-rational $p$-torsion point if and only if $E$ has split multiplicative reduction and the Tate period $q$ is a $p$-th power. Now, from \cite[Theorem V.5.3(b)]{SilvermanAdvanced}) we have that $E$ is split multiplicative if and only if $E\simeq E_q$ over $K$. Hence, in particular $E[p]\simeq E_q[p]$ (over $K$). On the other hand, if $E$ has non-split multiplicative reduction there are no $K$-rational $p$-torsion points. The reason for this is that we need to base change $E$ to an unramified quadratic extension of $K$ to get an isomorphism between $E$ and $E_q$. The details can be found in \cite[Section V.5]{SilvermanAdvanced}. 

Since $\qq\mid [F:L]$, proposition \ref{prop:arithmJxram} (iii) shows that, $F/L$ is ramified and $\Xscr_{\eJac_x}\otimes_L k$ is either $\Xscr_{\mathbf{A}^1_{k(\pp)}}$ or $\mathbb{A}^3_k$, depending on whether $\qq\mid x$ or not. 
\end{proof}

The above theorem suggests that $\eJac_x$ might not be the best algebra to use to study the case $\ell=p$. Therefore we will turn to the infinitesimal algebras from section \ref{sec:inf_envel}.  We remark that the infinitesimal algebras $\mathbf{I}_q$ enjoy the same ring-theoretic and homological properties as $\eJac_x$ as given in theorem \ref{prop:centre_J} and corollary \ref{cor:Jacks_hom}.

Assume that $K$ is a local field, unramified over $\Q_p$ (so $e=1$). We can (and will) choose $p$ as the uniformiser of $\oo_K$. Put $k:=\oo_K/p$. Because $K$ is unramified over $\Q_p$ we can write  
$$K=\Q_p(\zeta_m),\quad m=p^n-1, \quad n=[K:\Q_p]$$where $\zeta_m$ is a primitive $m$-th root of unity. Hence $k=\F_{p^n}$ and $\bar\zeta_m\in k$. 

Let $E_{/K}$ be an elliptic curve over $K$ and $\Ecal:=\Ecal_{/\oo_K}$ its minimal regular model. The extension $F/L=K(E[p])/K(\zeta_p)$ is a separable extension with ring of integers $\oo_L[\sqrt[p]{x}]=\oo_K[\zeta_p][\sqrt[p]{x}]$. Therefore, we can find a $\beta\in F$ such that $F/L=K(\zeta_p)(\beta)/K(\zeta_p)$ and $\oo_F:=\oo_L[\beta]$. 

We can choose $x$ as
$x=\sum_{i=0}^{p-1}\zeta^i_p \tau^i(\beta),$ where we have chosen a generator $\tau$ for $\mathrm{Gal}(F/L)$. The extension $K(E[p])/K(\zeta_p)=L(\beta)/L$ is totally (and wildly) ramified since reducing modulo the maximal ideal $\mm=(p)\subset\oo_K$ we have
$$\oo_L[\sqrt[p]{x}]=\oo_K[\zeta_p][\sqrt[p]{x}]\twoheadrightarrow k[T]/(T-\bar x)^p.$$In addition, $\bar x$ is identically zero since 
$$\bar x=\sum_{i=0}^{p-1}\bar\zeta^i_p \bar\tau^i(\bar\beta)=p\bar\beta=0.$$ Recall that, since the extension is ramified, $\bar\tau=\id$. Since $K(E[p])/K(\zeta_p)$ is a totally ramified Kummer extension it is, a priori, clear that $x$ must be a unit times a uniformiser (which in this case is $p$). 


The choice of $\mathtt{P}\in E[p](K)$ allows us to choose a $\mathtt{Q}\in E[p]$ such that $\{\mathtt{P}, \mathtt{Q}\}$ is a $\Z$-basis for $E[p]$. The lifting $\{\hat{\mathtt{P}}, \hat{\mathtt{Q}}\}$ becomes a $\Z$-basis on $\Ecal$, but its specialisation to the closed fibre need not be a basis. Since $\mathtt{P}\in E[p](K)$ the specialisation $\bar{\mathtt{P}}$ to the closed fibre can be viewed as an element in the special fibre $\bar\Ecal(k)$. 

Recall that $\oo_L^!=\oo_L[T]/(T^p)$ and consider the polynomial ring $\oo_L^![y]=\oo_L[T, y]/(T^p)$. There are two natural morphisms
$\oo_L^![y]\to L^![y]$ and $\oo_L^![y]\to k^![y]$. On $\oo_L^![y]$ we introduce the $\oo_L[y]$-linear automorphism $\sigma(T):=y T$. This naturally extends to the generic fibre $L^![y]$ and the special fibre $k^![y]$.

Evaluation $y\to\lambda\in K^\mathrm{al}$ defines an $\oo^!_L[\lambda]$-rational point on $\Spec(\oo_L^![y])$. We now construct the algebra
$$\mathbf{I}_\lambda:=\mathbf{E}\Big(\oo_L[\lambda]^!\cdot(\id-\sigma)\Big)=\frac{\oo_L[\lambda]\langle\eps_0,\eps_1,\dots,\eps_{p-1}\rangle}{\Big(\eps_i\eps_j-\lambda^{j-i}\eps_j\eps_i-(1-\lambda^{j-i})\eps_{i+j}\Big)},$$over $\oo_L[\lambda]$. Remember that we take $i+j$ modulo $p$.

There is a natural subalgebra of $\mathbf{I}_\lambda$, namely, 
$$\Inf_{\!\lambda}:=\oo_L[\lambda]\langle \eps_0, \eps_1, \eps_{p-1}\rangle\Bigg/
\begin{pmatrix}\eps_0\eps_1-\lambda\eps_1\eps_0-(1-\lambda)\eps_{1}\\
\eps_0\eps_{p-1} - \lambda^{p-1}\eps_{p-1}\eps_0-(1-\lambda^{p-1})\eps_{p-1}\\
\eps_1\eps_{p-1} - \lambda^{p-2}\eps_{p-1}\eps_1
\end{pmatrix}.$$ For simplicity, we put $\eps_2:=\eps_{p-1}$. If we want to make the prime explicit in the notation we write $\Inf_{\!\lambda}^{(p)}$. 

Making the substitution $\eps_0\mapsto \eps_0 +1$ and some re-arranging yields the isomorphic algebra (which we also denote $\Inf_{\!\lambda}$):
$$\Inf_{\!\lambda}=\oo_L[\lambda]\langle \eps_0, \eps_1, \eps_2\rangle\Bigg/
\begin{pmatrix}\eps_0\eps_1-\lambda\eps_1\eps_0\\
\eps_2\eps_{0} - \lambda^{1-p}\eps_{0}\eps_2\\
\eps_2\eps_{1} -\lambda^{2-p}\eps_{1}\eps_2
\end{pmatrix}.$$ 

\begin{remark}\label{rem:q-plane}
The algebra $\Inf_{\!\lambda}$ can be viewed as the glueing of three (in general) distinct quantum planes. 
\end{remark}

\subsubsection{The special fibre of $\Inf$}

We will now reduce modulo the maximal ideal $(p)$ and the generic point. The following diagram summarises the situation:
$$\xymatrix{\bar\Ecal\ar@{..}[d]&\Ecal\ar[l]_<<<<<<{\mathrm{red}}\ar[r]\ar@{..}[d]&E_{/K}\ar@{..}[d]\\
\Inf_{\!\lambda}\otimes_{\oo_K} k&\Inf_{\!\lambda}\ar[l]_<<<<{\mathrm{red}}\ar[r]&\Inf_{\!\lambda}\otimes_{\oo_K}K.}$$Put $\overline{\Inf}_{\!\lambda}=\Inf_{\!\lambda}\otimes_{\oo_K} k$, the \emph{special fibre} of $\Inf_{\!\lambda}$. We call $\Inf_{\!\lambda}\otimes_{\oo_K}K$, the \emph{generic fibre} (of course).

We will continue working with the minimal regular model $\Ecal$ of $E_{/K}$. Recall that $E(K)=\Ecal(\oo_K)$, and thus $E(K)[p]=\Ecal(\oo_K)[p]$. The set of non-singular $k$-points on $\bar E$ is isomorphic to the identity component of the special fibre $\bar\Ecal^0(k)$ (see \cite[Corollary IV.9.2(c)]{SilvermanAdvanced}). 
%


Fix $a, b\in\Z$ and let $$\boldsymbol{\theta}_{(a,b);S}:=S\big([a]\mathtt{P}+[b]\mathtt{Q}\big)\in F=K(E[p]),$$ where $S$ is a function in the coordinates of $[a]\mathtt{P}+[b]\mathtt{Q}$. Normally we will consider $a, b$ and $S$ as fixed and so we simply write $\boldsymbol{\theta}$ for $\boldsymbol{\theta}_{(a,b);S}$ unless confusion can arise. We write $\boldsymbol{\theta}=\sum_{i=0}^{p-1}\theta_i\beta^i$, with $\theta_i\in L$.

From now on we let $d$ be the number of irreducible components of $\bar\Ecal(k)$, counted without multiplicities, and $f$ the conductor of $E$. Put $w:=f+d$ and evaluate $\lambda$ at $\zeta_w$, where $\zeta_{w}\in\mu_{w}$ is primitive. We choose $\zeta=\zeta_w$ such that $\bar\zeta$ is also primitive. In addition, we set $\Inf_{\!w}:=\Inf_{\!\zeta}$. Ogg's formula (see \cite[IV.11.1]{SilvermanAdvanced}) shows that  $w=v_p(\Delta^\mathrm{min}_E)+1$, where $\Delta^\mathrm{min}_E$ is the minimal discriminant of $E_{/K}$. 

\begin{remark}Note that we make three successive evaluations: $y\to \lambda\to \zeta_w$. We could certainly do the evaluation $y\to\zeta_w$ directly, but I feel that this misses the point in the sense that the value $\zeta_w$ for $\lambda$ is a specific \emph{choice}.
\end{remark}

When $w=f+d$ is irrelevant we sometimes write $\zeta$ instead of $\zeta_w$. Put $k_w:=k(\zeta_w)$ and let $\mathbf{z}:=(z_0, z_1, \dots, z_{p-1})^T$ be a vector in $\oo_L[\zeta]^p$. Recall that $\boldsymbol{\theta}:=\boldsymbol{\theta}_{(a,b);S}$ is fixed. Define the \'etale $L$-rational point $\mathfrak{N}_{\boldsymbol{\theta},\mathbf{z}}$ to be the family of one-dimensional $\Inf_{\!w}$-modules $\mathfrak{N}_{j}:=\oo_L[\zeta]\cdot \mathbf{y}$ defined by
$$\eps_0\cdot\mathbf{y}:=\Delta_E \mathbf{y}, \quad \eps_1\cdot\mathbf{y}:=z	_j\mathbf{y}, \quad \eps_2\cdot \mathbf{y}:=\theta_j\mathbf{y},$$ where $\Delta_{E}$ is  the discriminant of $E$ and $\theta_j$ is the $j$-th coefficient of $\boldsymbol{\theta}$. To be explicit we will sometimes write $\mathfrak{N}_{j}$ as a triplet $(\Delta_{ E}, z_j, \theta_j)$ and the \'etale point $\mathfrak{N}_{\boldsymbol{\theta},\mathbf{z}}=(\Delta_E, \mathbf{z}, \boldsymbol{\theta})$.

 In order for $\mathfrak{N}_j$ to be an $\Inf_{\!w}$-module it is necessary and sufficient that 
$$(1-\zeta)\Delta_E z_j=0, \qquad (1-\zeta^{1-p})\Delta_E\theta_j=0, \quad \text{and}\quad (1-\zeta^{2-p})z_j\theta_j=0.$$ Clearly, the special fibre, which we denote $\overline{\Inf}_{\!w}$, is 
$$(1-\bar\zeta)\bar\Delta_E\bar z_j=0, \qquad (1-\bar\zeta^{1-p})\bar\Delta_E\bar\theta_j=0, \quad \text{and}\quad (1-\bar\zeta^{2-p})\bar z_j\bar\theta_j=0.$$


If $z_j\neq 0$ for all $j$ the \'etale rational point $\mathfrak{N}_{\boldsymbol{\theta},\mathbf{z}}$ is clearly empty since the element $\boldsymbol{\theta}$ is fixed. 
 
A computation shows the following theorem.
\begin{thm}\label{thm:ext_inf}
Let $\bar{\mathfrak{N}}_1$ and $\bar{\mathfrak{N}}_2$ be $(\bar\Delta_E, \bar z_1, \bar\theta_1)$ and $(\bar\Delta_E,\bar z_2, \bar\theta_2)$, respectively. Then 
$$T_\fam{N}=\Ext_{\overline{\Inf}_{\zeta}}^1(\bar{\mathfrak{N}}_1, \bar{\mathfrak{N}}_2)=\begin{cases}
k_w, & \text{if $\bar \Delta_E=0$ and $\bar z_1 = \bar z_2=0$ and $\bar\theta_2=\bar\zeta^{2-p}\bar\theta_1$}\\
k_w, & \text{if $\bar\Delta_E = 0$ and $\bar\theta_1=\bar\theta_2=0$ and $\bar z_1 = \bar\zeta^{2-p}\bar z_2$}\\
0, &\text{otherwise.}
\end{cases} $$From this follows that 
$$\hat\OO_{\fam{N}}=\begin{pmatrix}
	\End_{k_w}(\bar{\mathfrak{N}}_1)\otimes {k_w}[[t_{11}]] & \Hom_{k_w}(\bar{\mathfrak{N}}_1,\bar{\mathfrak{N}}_2)\otimes\langle t_{12}\rangle\\
	\Hom_{k_w}(\bar{\mathfrak{N}}_2,\bar{\mathfrak{N}}_1)\otimes\langle t_{21}\rangle &\End_{k_w}(\bar{\mathfrak{N}}_2)\otimes {k_w}[[t_{22}]]
\end{pmatrix}$$where $\fam{N}=\{\bar{\mathfrak{N}}_1, \bar{\mathfrak{N}}_2\}$.
\end{thm}
\begin{remark}The theorem applies also to the case $1-\bar\zeta^{2-p}=0$. Unless $z_1=z_2=0$ or $\theta_1=\theta_2=0$, we always end up in the third case.
\end{remark}
 
Assume $\bar z_j=0$ and let $\bar{\mathfrak{N}}_j$ be $(0,0, \bar\theta_j)$. Then the \emph{mirror} of $\bar{\mathfrak{N}}_j$ is the module $\bar{\mathfrak{N}}_j^\perp$ such that $\bar{\mathfrak{N}}_j^\perp=(0, 0, \bar\zeta^{2-p}\bar\theta_j)$.  Theorem \ref{thm:ext_inf} the shows that there is a tangent between $\bar{\mathfrak{N}}_j$ and $\bar{\mathfrak{N}}_j^\perp$, visually $\bar{\mathfrak{N}}_j\to\bar{\mathfrak{N}}_j^\perp$. Note that $(\bar{\mathfrak{N}}_j^\perp)^\perp$ is not necessarily $\bar{\mathfrak{N}}_j$. However, identifying the modules $\bar{\mathfrak{N}}_j=(0,0, \bar\zeta^{2-p}\bar\theta_j)$ and $\bar{\mathfrak{N}}_j'=(0,0, \bar\zeta^{p-2}\bar\theta_j)$, sets up a symmetry and so with this we can claim the implication $\bar{\mathfrak{N}}_j\to\bar{\mathfrak{N}}_j^\perp\,\,\Longrightarrow\,\, \bar{\mathfrak{N}}_j^\perp\to\bar{\mathfrak{N}}_j$. 

The \emph{mirror} of 
the \'etale $k_w$-rational point $\bar{\mathfrak{N}}_{\boldsymbol{\theta},\mathbf{z}}$ is naturally the family of modules $$\bar{\mathfrak{N}}_{\boldsymbol{\theta},\mathbf{z}}^\perp:=\bar\zeta^{2-p}\bar{\mathfrak{N}}_{\boldsymbol{\theta},\mathbf{z}}.$$ We say that $\mathfrak{N}_{\boldsymbol{\theta},\mathbf{z}}^\perp$ is \emph{a} mirror of $\mathfrak{N}_{\boldsymbol{\theta},\mathbf{z}}$ if $\bar{\mathfrak{N}}_{\boldsymbol{\theta},\mathbf{z}}^\perp$ is the mirror of $\bar{\mathfrak{N}}_{\boldsymbol{\theta},\mathbf{z}}$. Any mirror of $\mathfrak{N}_{\boldsymbol{\theta},\mathbf{z}}$ defines an element in $F$, namely $\boldsymbol{\theta}^\perp:=\zeta^{2-p}\boldsymbol{\theta}$, hence also in $\mathbf{I}_\zeta$, and by restriction, in $\Inf_{\!w}$. 

\subsubsection{$p\neq 3$}For $p\geq 5$ we have
$$f=\begin{cases}
0, & \text{if $E$ has good reduction,}\\
1, & \text{if $E$ has multiplicative reduction, and }\\
2, & \text{if $E$ has additive reduction.}\end{cases}$$We can directly note that, in the case of good reduction, $\overline{\Inf}_{\!0}^{(p)}=k_0[\eps_0, \eps_1, \eps_2]$. 

In the case of multiplicative reduction, i.e., $f=1$, the relations are
$$\overline{\Inf}_{\!d+1}^{(p)}:\quad \eps_0\eps_1-\bar\zeta_{d+1}\eps_1\eps_0, \quad \eps_2\eps_0 - \bar\zeta_{d+1}^{1-p}\eps_0\eps_2, \quad\eps_2\eps_1 - \bar\zeta_{d+1}^{2-p}\eps_1\eps_2
,$$where $d=-v_p(j_E)$.  Hence the reduction is a quantum affine space. 

When the reduction is additive we find the relations
$$\overline{\Inf}_{\!d+2}^{(p)}:\quad \eps_0\eps_1-\bar\zeta_{d+2}\eps_1\eps_0, \quad \eps_2\eps_0 - \bar\zeta_{d+2}^{1-p}\eps_0\eps_2, \quad\eps_2\eps_1 - \bar\zeta_{d+2}^{2-p}\eps_1\eps_2.$$The possible values for $d$ in the additive case are $$d=1, 2, 3, 5, 5-v_p(j_E), 7, 8, 9.$$However, it will turn out that not all these are possible when the curve has a $K$-rational $p$-torsion point.  

%
%
\subsubsection{$p=3$}
Due to possible wild ramification when $p=3$, the results are more subtle. Theorem IV.10.4 in \cite{SilvermanAdvanced} gives the bound $2\leq f\leq 5$, with each value possible. When $f>2$, \cite[Table IV.4.1]{SilvermanAdvanced}, shows that, when $p=3$, $d$ has the possibilities $d=1, 3, 7, 9$, corresponding to Kodaira types $\mathrm{II}$, $\mathrm{IV}$, $\mathrm{IV}^\ast$ and $\mathrm{II}^\ast$, respectively.  

We refrain from writing out all possible algebras, simply writing $\overline{\Inf}_{\!f+d}^{(3)}$ as given by the relations
$$\eps_0\eps_1-\bar\zeta_{f+d}\eps_1\eps_0, \quad \eps_2\eps_0 - \bar\zeta_{f+d}^{-2}\eps_0\eps_2, \quad\eps_2\eps_1 - \bar\zeta_{f+d}^{-1}\eps_1\eps_2, \quad $$ with $3\leq f\leq 5$ and  $d\in\{1, 3, 7, 9\}$. The case $f=2$ was treated above. 


%
%
%
%

\subsubsection{Multiplicative reduction}
Now, when the reduction is multiplicative the number of components is $d=-v_p(j(E))>1$, corresponding to the Kodaira type $\mathrm{I}_d$.

For the two examples below we fix $S$ and $(a, b)$  and write $\boldsymbol{\theta}=\boldsymbol{\theta}_{(a,b);S}=\sum_0^{p-1}\theta_i\beta^i$, with $\theta_i\in L$. I tried to find explicit curves with $K$-rational points so that some choice of $\boldsymbol{\theta}$ would be manageable, i.e., would be possible to write down in a reasonable way, but did not succeed.  

\begin{example}Suppose that $E$ has Kodaira type $\mathrm{I}_2$. Then
$$\overline{\Inf}_{\!\mathrm{I}_2}:=\frac{k(\bar\zeta_3)\langle \eps_0, \eps_1, \eps_2\rangle}{
\Big(\eps_0\eps_1-\bar\zeta_3\eps_1\eps_0, \,\,\eps_2\eps_0 - \bar\zeta_3^{1-p}\eps_0\eps_2,\,\,\eps_2\eps_1 -\bar\zeta_3^{2-p} \eps_1\eps_2
\Big)}.$$The centre $\cent(\overline{\Inf}_{\!\mathrm{I}_2})$ includes the algebra $k(\bar\zeta_3)[\eps_0^3, \eps_1^3, \eps_2^3]$, so $\Spec\big(\cent(\overline{\Inf}_{\!\mathrm{I}_2})\big)$, and by convention, $\Xscr_{\overline{\Inf}_{\!\mathrm{I}_2}}$, can be seen as an $\mathbb{A}^3$-fibration. 

Assume given a $\mathbf{z}$ such that $\bar\theta_i$  is not zero. Put $\bar{\mathfrak{N}}_i=(0, 0, \bar\theta_i)$.
From theorem \ref{thm:ext_inf} we see that 
$$T_{\fam{N}_i}=\Ext_{\overline{\Inf}_{\!\mathrm{I}_2}}^1\!\!(\bar{\mathfrak{N}}_i, \bar{\mathfrak{N}}')=\begin{cases}
k & \text{if $\bar\theta' = \bar\zeta^{2-p}_3\bar\theta_i$}\\
0, &\text{otherwise,}
\end{cases},$$where, of course, $\bar{\mathfrak{N}}'=(0,0,\bar{\boldsymbol{\theta}}')$ and $\fam{N}_i:= \{\bar{\mathfrak{N}}_i, \bar{\mathfrak{N}}'\}$. Hence $\bar{\mathfrak{N}}'=\bar{\mathfrak{N}}_i^\perp$. Notice that if $p=3$, $\bar\zeta_3=1$ so $\overline{\Inf}_{\!\mathrm{I}_2}^{(3)}=\mathbb{A}^3$. 
\end{example}

\begin{example}For $E$ of Kodaira type $\mathrm{I}_5$, we find
$$\overline{\Inf}_{\mathrm{I}_5}=\frac{k(\bar\zeta_{6})\langle \eps_0, \eps_1, \eps_2\rangle}{
\Big(\eps_0\eps_1-\bar\zeta_6\eps_1\eps_0, \,\,\eps_2\eps_0 -\bar\zeta_3^{1-p} \eps_0\eps_2,\,\,\eps_2\eps_1 +\bar\zeta_3^{2-p} \eps_1\eps_2
\Big)}.$$The slightly different relations here follow from $p$ being an odd prime and $\bar\zeta_6=\bar\zeta_2\bar\zeta_3$. 

In this case the centre is quite complicated but it certainly includes the algebra $k(\bar\zeta_6)[\eps_0^6, \eps_1^6, \eps_2^6]$ (hence $\Spec\big(\cent(\overline{\Inf}_{\!\mathrm{I}_2})\big)$ and $\Xscr_{\overline{\Inf}_{\!\mathrm{I}_2}}$ can be seen as $\mathbb{A}^3$-fibrations in this case also). With the same notation as in the previous example we find 
$$T_{\fam{N}_i}=\Ext_{\overline{\Inf}_{\mathrm{I}_5}}^1\!\!(\bar{\mathfrak{N}}_i, \bar{\mathfrak{N}}')=\begin{cases}
k & \text{if $\bar\theta' = -\bar\zeta^{2-p}_3\bar\theta_i$}\\
0, &\text{otherwise,}
\end{cases} $$following from theorem \ref{thm:ext_inf}.
\end{example}

The case of additive reduction is more subtle as was alluded to above.  The possible number of components are $1\leq d\leq 9$ but it will turn out that not all of these are possible. 

\subsubsection{Additive reduction}
We will look at bit deeper into the case of additive reduction. The following theorem is \cite[Theorem 12]{KostersPannekoek}. 

\begin{thm} \label{thm:KostersPannekoek}Let $E_{/K}$ be an elliptic curve with additive reduction at $\qq_K\mid p$. Then
\begin{itemize}
	\item[(i)] $E_0(K)\simeq \widehat{E}(\oo_K)$ as (topological) groups, and
	\item[(ii)] if $6e(\qq)<p-1$, then $E_0(K)\simeq \Z_p^n$, as $\Z_p$-modules. When $6e(\qq)\geq p-1$, we have $E_0(K)\simeq \Z_p^n\times \Z/p$. 
\end{itemize}Here $n$ is the degree of $K/\Q_p$. 
\end{thm}
\begin{thm}\label{thm:addreduction}
Let $E_{/K}$ be an elliptic curve with a $K$-rational $p$-torsion point and $e(\qq)=1$. Then  $E_{/K}$ cannot have additive reduction at $\qq_K$, unless $p=2, 3, 5$ or $7$.
\end{thm}
\begin{proof}By theorem \ref{thm:KostersPannekoek} (ii), if $e(\qq)=1$, then $E_0(K)\simeq \Z_p^n$ when $p\geq 11$. Hence the only possibilities to have $p$-torsion in the additive reduction case are when $p=2, 3, 5, 7$. 
\end{proof}

Let $E_{/K}$ be an elliptic curve with additive reduction at $\qq_K$ and with Weierstrass equation
\begin{equation}\label{eq:WeierKP}
E_{/K}:\quad y^2+a_1 xy + a_3y = x^3+a_2x^2+a_4x+ a_6.
\end{equation} In fact, \cite[Lemma 9]{KostersPannekoek} shows that the coefficients can be chosen so that they are all in the maximal ideal $\mathfrak{m}\subset \oo_K$. 

Recall that, for an elliptic curve $E_{/K}$, there is an exact sequence 
$$\xymatrix{0\ar[r]&E_1(K)\ar[r]&E_0(K)\ar[r]^>>>>>{\mathrm{red}}&\bar E_{\mathrm{sm}}(k)\ar[r]& 0},$$where $E_0(K)$ is the set of $K$-rational points on $E$ reducing to the smooth locus $\bar E_{\mathrm{sm}}$ on the reduced curve $\bar E$ and where $E_1(K)$ is the kernel of reduction. Also, $E_1(K)\simeq \widehat{E}(\oo_K)$, where $\widehat{E}$ is the formal group of $E$. See \cite[Propositions 2.1 and 2.2, chapter VII]{SilvermanElliptic} for details. 
 Let $E_{/K}$ have additive reduction and let $\Ecal_{/\oo_K}$ be its minimal regular model. Then $$\bar\Ecal^0(k)\simeq \bar{E}_\mathrm{sm}(k),$$ where $\bar\Ecal^0$ is the connected component of the identity in the closed fibre, by \cite[IV.9.2(c)]{SilvermanAdvanced}. Also, there is an exact sequence of group schemes
\begin{equation}\label{eq:neroncomp}
\xymatrix{0\ar[r]&\Ecal^0\ar[r]&\Ecal\ar[r]&\Phi.}
\end{equation}From this sequence and \cite[IV.9.2(a,b)]{SilvermanAdvanced} follows the exact sequences
$$\xymatrix{0\ar[r]&\Ecal^0(\oo_K)\ar[r]&\Ecal(\oo_K)\ar[r]&\Phi(\oo_K)\ar[r]&0\\
0\ar[r]&E_0(K)\ar@{=}[u]\ar[r]&E(K)\ar@{=}[u]\ar[r]&\Phi(K)\ar@{=}[d]\ar@{=}[u]\ar[r]&0\\
&&&E(K)/E_0(K)}$$and
$$\xymatrix@C=18pt{0\ar[r]&\Ecal^0(\oo_K)[p^\infty]\ar[r]&\Ecal(\oo_K)[p^\infty]\ar[r]&\Phi(\oo_K)[p^\infty]\\
0\ar[r]&E_0(K)[p^\infty]\ar@{=}[u]\ar[r]&E(K)[p^\infty]\ar@{=}[u]\ar[r]&\Phi(K)[p^\infty]\ar@{=}[r]\ar@{=}[u]&\big(E(K)/E_0(K)\big)[p^\infty].}$$
In addition, \cite[IV.9.2(b)]{SilvermanAdvanced} gives that 
$$\Phi(K)=E(K)/E_0(K)\simeq \bar\Ecal^0(k)=\Phi(k).$$ The group $\Phi(k)$ is the group of components of $\bar\Ecal$. 

Now, table \cite[Table IV.4.1]{SilvermanAdvanced} shows that $\Phi(k)[p]=\{0\}$ if $p>3$ and hence $E_0(K)[p]\simeq E(K)[p]$. Therefore, all $K$-rational $p$-torsion points must reduce to non-singular points on $\bar{E}(k)=\bar\Ecal(k)$. Theorem \ref{thm:KostersPannekoek} now implies that, if $p>7$, $E_0(K)$ is $p$-torsion-free, and hence also $E(K)$. 

Let us begin by looking at the cases $p=5$ or $7$. Since $E(K)[p]\simeq E_0(K)[p]$, the discussion in section 3.3.1 in \cite{KostersPannekoek} implies that $$E(K)[p]\simeq \Z_p^n\times\Z/p,\qquad\text{(with $n=[K:\Q_p]$)}$$ if and only if $f(T):=T-aT^p$ has a non-trivial solution in $k$, where $a=3a_4/5 \,\, (\text{mod $\qq$})$ when $p=5$ and $a=4a_6/7 \,\, (\text{mod $\qq$})$, when $p=7$, and where $a_4$, $a_6$ are the coefficients given by the Weierstrass model (\ref{eq:WeierKP}). See \cite[Section 3.3.1]{KostersPannekoek} for the details. It is important to be aware that these congruences refer \emph{only} to the Weierstrass model (\ref{eq:WeierKP}) where the coefficients are in $\mathfrak{m}$. Incidentally, the polynomial $f(T)$ is the reduction of the power series $[p](T)$ in the formal group of $E_{/K}$.

If a solution to the these congruences does not exist, $E(K)[p]=\Z_p^n$ and hence there are no non-trivial $K$-rational $p$-torsion points. 
The above comment indicates that the number of curves with additive reduction with a $K$-rational $p$-torsion point is quite rare for $p=5$ and rarer still for $p=7$. In fact, the first curve where the curve has additive reduction with a non-trivial $\Q_7$-rational point is \textsf{294b2} in Cremona's table \cite{Cremona}. The next curve after this  is the curve \textsf{490k2}.

The following two examples show that existence and non-existence of rational $p$-torsion points can occur for both $p=5$ and $p=7$ in the additive reduction case.  The computations are done with Sage \cite{sagemath}.
\begin{example}Let $E_{/\Q_5}$ be the elliptic curve given by the Weierstrass equation
\begin{equation}\label{eq:50b2}
y^2 + xy+y = x^3+x^2+22x-9. \tag{\text{\textsf{50b2}}}
\end{equation}
This curve has bad additive reduction (of type $\mathrm{II}$) and is \textsf{50b2} in Cremona's table \cite{Cremona}. The $\Q_5$-rational $5$-torsion points are
$$\Big\{\infty,\, (1 : -5 : 1),\, (1 : 3 : 1),\, (9 : -37 : 1),\, (9 : 27 : 1)\Big\}$$ and so $E(\Q_5)[5]=E(\Q)[5]\simeq \Z/5$. All these have non-singular reduction. On the other hand, the curve \textsf{50a4} also has bad additive reduction (of type $\mathrm{IV}^\ast$). However, this time the only rational $5$-torsion point is $\{\infty\}=\{[0:1:0]\}$. 
\end{example}
\begin{example}Now, let $E_{/\Q_7}$ be the elliptic curve \textsf{49a3} given by the Weierstrass equation
\begin{equation}\label{eq:49a3}
y^2 + xy = x^3 - x^2 - 107x + 552. \tag{\text{\textsf{49a3}}}
\end{equation}
This curve has bad additive reduction of type $\mathrm{III}^\ast$. The only $\Q_7$-rational $7$-torsion points is $\{\infty\}=\{[0:1:0]\}$. Looking now at the curve $E'_{/\Q_7}$
\begin{equation}\label{eq:490k2}
y^2 + xy + y = x^3 - x^2 + 918x + 5289 \tag{\text{\textsf{490k2}}}
\end{equation}of bad additive Kodaira type $\mathrm{II}$, one finds that 
\begin{multline*}E'(\Q_7)[7]=E'(\Q)[7]=\Big\{\infty,\, (-3 : -49 : 1),\, (-3 : 51 : 1),\, (17 : -169 : 1),\\
 (17 : 151 : 1),\, (97 : -1049 : 1),\, (97 : 951 : 1)\Big\},
 \end{multline*}all of non-singular reduction.
\end{example}

When $p=5$ or $7$, the conductor is $f=2$ for type $\mathrm{II}$ and $\mathrm{III}$. In the case $p=7$ the only possibility is $d=1$ (Kodaira symbol $\mathrm{II}$) and so, since all $3$-rd roots of unity are in $\F_7$,     
$$\overline{\Inf}_{\mathrm{II}}^{(7)}=\frac{k\langle \eps_0, \eps_1, \eps_2\rangle}{
\Big(\eps_0\eps_1-\bar\zeta_3\eps_1\eps_0, \,\,\eps_2\eps_0 - \eps_0\eps_2,\,\,\eps_2\eps_1 -\bar\zeta_3^{-2} \eps_1\eps_2
\Big)}.$$When $p=5$ we find the algebras 
$$\overline{\Inf}_{\mathrm{III}}^{(5)}=\frac{k\langle \eps_0, \eps_1, \eps_2\rangle}{
\Big(\eps_0\eps_1-\bar\zeta_4\eps_1\eps_0, \,\,\eps_2\eps_0 - \eps_0\eps_2,\,\,\eps_2\eps_1 +\bar\zeta_4 \eps_1\eps_2
\Big)},$$ where $\bar\zeta_4\in \F_5$, and
$$\overline{\Inf}_{\mathrm{II}}^{(5)}=\frac{k(\bar\zeta_3)\langle \eps_0, \eps_1, \eps_2\rangle}{
\Big(\eps_0\eps_1-\bar\zeta_3\eps_1\eps_0, \,\,\eps_2\eps_0 -\bar\zeta_3^{-1} \eps_0\eps_2,\,\,\eps_2\eps_1 - \eps_1\eps_2
\Big)}.$$ Note that, depending on the degree of $k/\F_5$, $\bar\zeta_3$ may or may not be in $k$ already. The degree $[k(\bar\zeta_3):k]$ is either $1$ or $2$.

In the case $p=3$, the group $\Phi(k)=\Phi(K)$ is not necessarily trivial. Therefore, we cannot conclude from (\ref{eq:neroncomp}) that 
$$E_0(K)[3]\simeq E(K)[3].$$ In fact, this is false in general. Thanks to Chris Wuthrich for indicating the following examples. Computations once again courtesy of Sage \cite{sagemath}.
\begin{example}\label{exam:27a1}Let $E_{/\Q_3}$ be the elliptic curve
\begin{equation}\label{eq:27a1}
y^2+y=x^3-7\tag{\text{\textsf{27a1}}}
\end{equation}
over $\Q_3$. Then $$E(\Q_3)[3^\infty]\simeq \Z/3$$ and only $\{\infty\}=\{[0:1:0]\}$ has good reduction. In fact, $E(\Q)[3]=E(\Q_3)[3]=E(\Q_3)[3^\infty]$. This curve has Kodaira type $\mathrm{IV}^\ast$.
\end{example}
\begin{example}Now, let $E_{/\Q_3}$ be Cremona's curve \textsf{54b3}: 
\begin{equation}\label{eq:54b3}
y^2+xy+y = x^3-x^2-14x+29. \tag{\text{\textsf{54b3}}}
\end{equation}
In this example
$$E(\Q_3)[3^\infty]\simeq \Z/9$$ and three $\Q_3$-rational points have good reduction (incidentally, these are the points of $E(\Q)[3]$). The Kodaira symbol for \textsf{54b3} is $\mathrm{IV}$. 
\end{example}

The possibility for the $p$-torsion points to reduce into the singular locus is measured by $\Phi(k)=\Phi(K)=E(K)/E_0(K)$ and when $p=3$ the only possibility is $\Phi(K)=\Z/3$, which corresponds to the Kodaira symbols $\mathrm{IV}$ and $\mathrm{IV}^\ast$. The reason for this is that $E(K)[3]/E_0(K)[3]$ is a subgroup of $\Phi(K)$ of order $3$. Hence $3$ must be a divisor of the order of $\Phi(K)$. According to \cite[Table IV.4.1]{SilvermanAdvanced} this corresponds to $E(K)[3]/E_0(K)[3]=E(K)/E_0(K)=\Z/3$ and the Kodaira symbols $\mathrm{IV}$ and $\mathrm{IV}^\ast$. 

Another possibility is that $E(K)[3]/E_0(K)[3]$ is trivial while $E(K)/E_0(K)=(0)$ or $E(K)/E_0(K)=\Z/2$. In this case we get the symbols $\mathrm{II}$ and $\mathrm{III}$. This happens for instance for \textsf{54b1} (which is of type $\mathrm{II}$) and \textsf{90b2} (which is of type $\mathrm{III}$). Note that $E(K)[3]/E_0(K)[3]$ and $E(K)/E_0(K)=(0)$ can both be trivial without $E$ having good reduction (when $\bar E$ is a cusp with only one component). 

On the other hand, the Kodaira symbols I have found for the $3$-torsion reducing to the smooth locus, for curves over $\Q_3$, are $\mathrm{II}$, $\mathrm{III}$, $\mathrm{IV}$, $\mathrm{I}_0^\ast$, $\mathrm{I}_1^\ast$, $\mathrm{I}_2^\ast$, $\mathrm{I}_3^\ast$ and $\mathrm{I}_4^\ast$. At the moment I'm not sure whether $\mathrm{I}_n^\ast$ can appear for $n\geq 5$ or not (and, in that case, if there is an upper bound). This shouldn't be too difficult to prove or disprove. 

\begin{remark}It is important to remember that all the above claims are made under the assumption that $E$ has a \emph{$K$-rational} $3$-torsion point. 
\end{remark}

Summarising the additive case when $p=3$:
\begin{thm}\label{thm:sum_additive}
Suppose $E_{/\Q_3}$ has additive reduction. The relations for $\overline{\Inf}_{\!f+d}^{(3)}$ are 
$$\eps_0\eps_1-\bar\zeta_{f+d}\eps_1\eps_0, \quad \eps_2\eps_0 - \bar\zeta_{f+d}^{-2}\eps_0\eps_2, \quad\eps_2\eps_1 - \bar\zeta_{f+d}^{-1}\eps_1\eps_2$$ and the possibilities for $d$ and $f$ are (remember, $2\leq f\leq 5$):
\begin{itemize}
	\item[$\mathrm{II}$:] $d=1$, $f>2$;
	\item[$\mathrm{III}$:] $d=2$, $f=2$;
	\item[$\mathrm{IV}$:] $d=3$, $f>2$;
	\item[$\mathrm{I}^\ast_n$:] $d=5+n$, $f=2$, and
	\item[$\mathrm{IV}^\ast$:] $d=7$, $f>2$.
\end{itemize}Depending on $d$, $f$ and $[k:\F_3]$, $\bar\zeta_{f+d}$ may, or may not, be in $k$ already. 
\end{thm}

\section{The $\theta\!\!\!\theta$-elements}\label{sec:theta_elements}
Now, fix $S$. We then get a function 
\begin{equation}\label{eq:beta_1}
\boldsymbol{\theta}_{-;S}:\,\,E[p]\longrightarrow F, \quad [a]\mathtt{P}+[b]\mathtt{Q}\,\,\,\longmapsto\,\,\, \boldsymbol{\theta}=\boldsymbol{\theta}_{(a,b); S}=\sum_{i=0}^{p-1}\theta_i\beta^i.
\end{equation} This defines a one-dimensional \'etale $L$-rational point of order $p-1$, $$\mathfrak{N}_{\boldsymbol{\theta}, \mathbf{0}}=\Big\{\mathfrak{N}_j\mid 0\leq j\leq p-1\Big\}, \quad \mathfrak{N}_j=(0,0, \theta_j),$$ on $\Xscr_{\Inf_{\!f+d}}$. We can define a mirror family $\mathfrak{N}^\perp_{\boldsymbol{\theta}, \mathbf{0}}$ such that $\bar{\mathfrak{N}}_{\boldsymbol{\theta}, \mathbf{0}}^\perp=\zeta^{2-p}\bar{\mathfrak{N}}_{\boldsymbol{\theta}, \mathbf{0}}$.

Hence,
\begin{prop}The association (\ref{eq:beta_1}) sets up a correspondence between $E[p]$ and one-dimensional \'etale $L$-rational points of order in $\Xscr_{\Inf_{\!f+d}}$.  Each such \'etale $L$-rational point $\mathfrak{N}_{\boldsymbol{\theta}, \mathbf{0}}$ defines a unique mirror $\bar{\mathfrak{N}}_{\boldsymbol{\theta}, \mathbf{0}}^\perp$ on the special fibre $\Xscr_{\Inf_{\!f+d}}\otimes k_w$.
\end{prop}

There is another way to use the correspondence (\ref{eq:beta_1}). It is natural to say that the $\Inf_{\!f+d}$-module $\mathfrak{H}:=\Inf_{\!f+d}/(H)$, where $H:=a_0\eps_0+a_1\eps_1+a_2\eps_2\in \Inf_{\!f+d}$, is a \emph{hyperplane} in $\Xscr_{\Inf_{\!f+d}}$.

\begin{prop}The association $\boldsymbol{\theta}_{-;S}^\mathfrak{H}:\,\,E[p]\longrightarrow \Inf_{\!f+d}$,
\begin{equation}\label{eq:beta_2}
[a]\mathtt{P}+[b]\mathtt{Q}\,\,\,\longmapsto\,\,\, \sum_{i=0}^{p-1}\theta_i\beta^i\,\,\,\longmapsto\,\,\, H:=\theta_0\eps_0+\theta_1\eps_1+\theta_{p-1}\eps_2.
\end{equation} 
(recall $\eps_2$ is actually $\eps_{p-1}$) sets up a correspondence between $E[p]$ and hyperplanes in $\Xscr_{\Inf_{\!f+d}}$.  Hence the function $\boldsymbol{\theta}_{-;S}^\mathfrak{H}$ is a \emph{hyperplane arrangement} in $\Xscr_{\Inf_{\!f+d}}$ defined by $E[p]$.
\end{prop}Note that the restriction of $\boldsymbol{\theta}_{E[p];S}^\mathfrak{H}$ to the central subscheme defines an actual hyperplane arrangement in $\Spec(\cent(\Inf_{\!f+d}))$. 

\begin{remark}The above constructions are obviously not restricted to the particular choice $w=f+d$. 
\end{remark}

As a third, and final, construction involving the $\boldsymbol{\theta}$-elements we present a way to construct elements in the Brauer group of $L[\zeta_{\!f+d}]$. 

As mentioned in remark (\ref{rem:q-plane}) we can view $\Inf_{\!w}$ as the glueing of three quantum planes 
$$\mathbf{Q}_{i, j}:=\oo_L[\zeta_{f+d}]\langle \eps_i, \eps_j\rangle \Big/ (\eps_i\eps_j-\zeta_{f+d}^a\eps_j\eps_i), \quad 0\leq i<j\leq 2,$$where $a = 1$, $p-1$ or $p-2$.

Given a hyperplane $\mathfrak{H}=\theta_0\eps_0+\theta_1\eps_1+\theta_{p-1}\eps_2$ we can form the following quotients
$$\mathbf{Q}_{i, j}^{\boldsymbol{\theta}}:=\frac{\mathbf{Q}_{i, j}}{\Big(\eps_i^{f+d}=\theta_i, \,\, \eps_j^{f+d}=\theta_j\Big)}, \quad 0\leq i<j\leq 2.$$If $\theta_i\theta_j\neq 0$, $\mathbf{Q}_{i, j}^{\boldsymbol{\theta}}$ defines an Azumaya algebra over $\oo_L[\zeta_{\!f+d}]$, defining an element in  $\mathrm{Br}(\oo_L[\zeta_{\!f+d}])$. Consequently the generic fibre is a central simple algebra and hence defines an element in $\mathrm{Br}(L[\zeta_{\!f+d}])$. By the functoriality of $\mathrm{Br}$ there is a natural morphism $ \mathrm{Br}(\oo_L[\zeta_{\!f+d}])\to \mathrm{Br}(L[\zeta_{\!f+d}])$. 
\begin{prop}The association $\boldsymbol{\theta}_{-;S}^\mathrm{Br}:\,\,E[p]\longrightarrow \mathrm{Br}(L[\zeta_{\!f+d}])$,
\begin{equation}\label{eq:beta_3}
[a]\mathtt{P}+[b]\mathtt{Q}\,\,\,\longmapsto\,\,\, \sum_{l=0}^{p-1}\theta_l\beta^l\,\,\,\longmapsto\,\,\, \mathbf{Q}_{i, j}^{\boldsymbol{\theta}}, \quad 0\leq i<j\leq 2, \quad \theta_i\theta_j\neq 0,
\end{equation} sets up a correspondence between $E[p]$ and elements in $\mathrm{Br}(L[\zeta_{\!f+d}])$, and by restriction to a (possibly trivial) element in $\mathrm{Br}(L)$.
\end{prop}

The above correspondences (\ref{eq:beta_1}), (\ref{eq:beta_2}) and (\ref{eq:beta_3}) might deserve some further study.
\begin{remark}If $f+d=2$ the algebras $\mathbf{Q}_{i, j}^{\boldsymbol{\theta}}$, under the condition $\theta_i\theta_j\neq 0$, are all quaternion algebras. This can only happen when $E$ has Kodaira symbol $\mathrm{I}_1$, i.e., when $\bar\Ecal=\mathbb{P}^1$. 
\end{remark}
\begin{center}
\rule{0.50\textwidth}{0.5pt}
\end{center}
\bibliographystyle{alpha}
\bibliography{ref_Arit_NC}

\begin{thebibliography}{CDyDO03}

\bibitem[BP16]{BandiniPaladino}
Andrea Bandini and Laura Paladino.
\newblock Fields generated by torsion points of elliptic curves.
\newblock {\em J. Number Theory}, 169:103--133, 2016.

\bibitem[Cas67]{Cassels_Frohlich}
{\em Algebraic number theory}.
\newblock Proceedings of an instructional conference organized by the London
  Mathematical Society (a NATO Advanced Study Institute) with the support of
  the International Mathematical Union. Edited by J. W. S. Cassels and A.
  Fr\"{o}hlich. Academic Press, London; Thompson Book Co., Inc., Washington,
  D.C., 1967.

\bibitem[CDyDO03]{Cohen_et_al_Cyclotomic}
Henri Cohen, Francisco Diaz~y Diaz, and Michel Olivier.
\newblock Cyclotomic extensions of number fields.
\newblock {\em Indag. Math. (N.S.)}, 14(2):183--196, 2003.

\bibitem[Cre15]{Cremona}
J.~Cremona.
\newblock Elliptic curve data for conductors up to 350000, 2015.

\bibitem[ELS17]{EriksenLaudalSiqveland}
Eivind Eriksen, Olav~Arnfinn Laudal, and Arvid Siqveland.
\newblock {\em Noncommutative deformation theory}.
\newblock Monographs and Research Notes in Mathematics. CRC Press, Boca Raton,
  FL, 2017.

\bibitem[FK18]{FreitasKraus}
Nuno Freitas and Alain Kraus.
\newblock On the degree of the $p$-torsion field of elliptic curves over
  $\mathbb{Q}_\ell$ for $\ell \neq p$, 2018.

\bibitem[GJLR16]{Jimenez-Lozano}
Enrique Gonz\'{a}lez-Jim\'{e}nez and \'{A}lvaro Lozano-Robledo.
\newblock Elliptic curves with abelian division fields.
\newblock {\em Math. Z.}, 283(3-4):835--859, 2016.

\bibitem[HLS06]{HaLaSi}
Jonas~T. Hartwig, Daniel Larsson, and Sergei~D. Silvestrov.
\newblock Deformations of {L}ie algebras using {$\sigma$}-derivations.
\newblock {\em J. Algebra}, 295(2):314--361, 2006.

\bibitem[Kid03]{Kida_Ramification_Division_fields}
M.~Kida.
\newblock Ramification in the division fields of an elliptic curve.
\newblock {\em Abh. Math. Sem. Univ. Hamburg}, 73:195--207, 2003.

\bibitem[KP17]{KostersPannekoek}
Michiel {Kosters} and Ren{\'e} {Pannekoek}.
\newblock {On the structure of elliptic curves over finite extensions of
  $\mathbb{Q}_p$ with additive reduction}.
\newblock {\em arXiv:1703.07888}, Mar 2017.

\bibitem[Lar17]{LarArithom}
Daniel Larsson.
\newblock Equivariant hom-{L}ie algebras and twisted derivations on
  (arithmetic) schemes.
\newblock {\em J. Number Theory}, 176:249--278, 2017.

\bibitem[Lar23a]{LarssonKummerWitt}
Daniel Larsson.
\newblock Kummer--{W}itt--{J}ackson algebras.
\newblock {\em Preprint}, August 2023.
\newblock Available at http://www.numberlab.se.

\bibitem[Lar23b]{LarssonAritGeoLargeCentre}
Daniel Larsson.
\newblock Arithmetic {G}eometry of {A}lgebras with {L}arge {C}entres.
\newblock {\em Preprint}, June 2023.
\newblock Available at http://www.numberlab.se.

\bibitem[Lau02]{Laudal_Def}
O.~A. Laudal.
\newblock Noncommutative deformations of modules.
\newblock {\em Homology Homotopy Appl.}, 4(2, part 2):357--396, 2002.
\newblock The Roos Festschrift volume, 2.

\bibitem[LS07]{LaSi}
Daniel Larsson and Sergei~D. Silvestrov.
\newblock Quasi-deformations of {${\mathfrak{sl}}\sb 2(\mathbb{F})$} using
  twisted derivations.
\newblock {\em Comm. Algebra}, 35(12):4303--4318, 2007.

\bibitem[{Sag}20]{sagemath}
{Sage Developers, The}.
\newblock {\em {S}ageMath, the {S}age {M}athematics {S}oftware {S}ystem
  ({V}ersion 9.0)}, 2020.
\newblock {\tt https://www.sagemath.org}.

\bibitem[Ser72]{Serre_Galois_Elliptic}
Jean-Pierre Serre.
\newblock Propri\'{e}t\'{e}s galoisiennes des points d'ordre fini des courbes
  elliptiques.
\newblock {\em Invent. Math.}, 15(4):259--331, 1972.

\bibitem[Sil86]{SilvermanElliptic}
Joseph~H. Silverman.
\newblock {\em The arithmetic of elliptic curves}, volume 106 of {\em Graduate
  Texts in Mathematics}.
\newblock Springer-Verlag, New York, 1986.

\bibitem[Sil94]{SilvermanAdvanced}
Joseph~H. Silverman.
\newblock {\em Advanced topics in the arithmetic of elliptic curves}, volume
  151 of {\em Graduate Texts in Mathematics}.
\newblock Springer-Verlag, New York, 1994.

\bibitem[Yas13]{Yasuda_Kummer}
Masaya Yasuda.
\newblock Kummer generators and torsion points of elliptic curves with bad
  reduction at some primes.
\newblock {\em Int. J. Number Theory}, 9(7):1743--1752, 2013.

\end{thebibliography}

\end{document}